\RequirePackage{fix-cm}
\documentclass{svjour3_noheadbox}       % onecolumn (second format)
\smartqed  % flush right qed marks, e.g. at end of proof

\usepackage{amsmath,amssymb,mathtools,dsfont,extarrows,mathrsfs}
\usepackage{enumitem,verbatim,graphicx}
\usepackage{algorithm,algpseudocode}
\usepackage{booktabs}
\usepackage[margin=1in]{geometry}
\newcommand{\E}{\mathbb{E}}

\newcommand{\cP}{\mathcal{P}}

\newcommand{\R}{\mathbb{R}}

\newcommand{\frakM}{\mathfrak{M}}

\newcommand{\supp}{\mathrm{supp}\ }
\newcommand{\conv}{\mathrm{conv}}
\newcommand{\sfd}{\mathsf{d}}
\newcommand{\sfW}{\mathsf{W}}

\newcommand{\bxi}{\boldsymbol{\xi}}
\newcommand{\hxi}{\hat{\xi}}

\newcommand{\hSigma}{\hat{\Sigma}}

\newcommand{\bzeta}{\boldsymbol{\zeta}}
\newcommand{\bmu}{\boldsymbol{\mu}}
\newcommand{\bnu}{\boldsymbol{\nu}}
\newcommand{\bgamma}{\boldsymbol{\gamma}}
\newcommand{\bdelta}{\boldsymbol{\delta}}
\newcommand{\bpi}{\boldsymbol{\pi}}
\newcommand{\bC}{\boldsymbol{\mathcal{C}}}

\DeclareMathOperator*{\argmax}{arg\,max}
\DeclareMathOperator*{\esssupOp}{-ess\,sup}
\newcommand\esssup[2]{\underset{#2}{#1\!\esssupOp}}
\spnewtheorem{observation}{Observation}{\bf}{\it}
\spnewtheorem*{theorem*}{Theorem}{\bf}{\it}
\spnewtheorem{definition*}{Definition}{\bf}{\it}

\begin{document}

\title{Distributionally Robust Stochastic Optimization with Dependence Structure%\thanks{Grants or other notes
}
\author{Rui Gao         \and
        Anton J. Kleywegt %etc.
}
\institute{Rui Gao \at
              School of Industrial and Systems Engineering, Georgia Institute of Technology, Atlanta, GA\\
              \email{rgao32@gatech.edu}           %  \\
           \and
           Anton J. Kleywegt \at
              School of Industrial and Systems Engineering, Georgia Institute of Technology, Atlanta, GA\\
              \email{anton@isye.gatech.edu}
}

\maketitle

\begin{abstract}
%Insert your abstract here. Include keywords, PACS and mathematical
%subject classification numbers as needed.
Distributionally robust stochastic optimization (DRSO) is a framework for decision-making problems under certainty, which finds solutions that perform well for a chosen set of probability distributions.
Many different approaches for specifying a set of distributions have been proposed.
The choice matters, because it affects the results, and the relative performance of different choices depend on the characteristics of the problems.
In this paper, we consider problems in which different random variables exhibit some form of dependence, but the exact values of the parameters that represent the dependence are not known.
We consider various sets of distributions that incorporate the dependence structure, and we study the corresponding DRSO problems.

In the first part of the paper, we consider problems with linear dependence between random variables.
We consider sets of distributions that are within a specified Wasserstein distance of a nominal distribution, and that satisfy a second-order moment constraint.
We obtain a tractable dual reformulation of the corresponding DRSO problem.
This approach is compared with the traditional moment-based DRSO, which considers all distributions whose first- and second-order moments satisfy certain constraints, and with the Wasserstein-based DRSO, which considers all distributions that are within a specified Wasserstein distance of a nominal distribution (with no moment constraints).
Numerical experiments suggest that our new formulation has superior out-of-sample performance.

In the second part of the paper, we consider problems with various types of rank dependence between random variables, including rank dependence measured by Spearman's footrule distance between empirical rankings, comonotonic distributions, box uncertainty for individual observations, and Wasserstein distance between copulas associated with continuous distributions.
We also obtain a dual reformulation of the DRSO problem.
A desirable byproduct of the formulation is that it also avoids an issue associated with the one-sided moment constraints in moment-based DRSO problems.

\end{abstract}

\section{Introduction}
\label{sec:intro}

Stochastic optimization is an approach to optimization under uncertainty with a well-developed foundation of theory and practical applications.
A core issue in stochastic optimization is that often the underlying probability distribution is not known, or the notion of multiple realizations from a single underlying probability distribution may be a questionable description of reality.
Distributionally robust stochastic optimization (DRSO) is an approach to optimization under uncertainty in which, instead of assuming that there is an underlying probability distribution that is known to the optimizer, one finds a decision $x \in X$ that provides the best hedge against a set of probability distributions, by solving the following problem:
\begin{equation}
\label{eqn:DRSO}
\inf_{x \in X} \sup_{\bmu \in \frakM} \E_{\bmu}[\Psi(x,\bxi)],
\end{equation}
where the cost function $\Psi : X \times \Xi \mapsto \R$ depends on a random quantity $\bxi$ which takes values in $\Xi \subset \R^{K}$, and $\frakM$ is a subset of the set $\cP(\Xi)$ of all Borel probability distributions on $\Xi$.

The set $\frakM$ of probability distributions is chosen to make the resulting decisions robust against future variations in $\bxi$.
% typically guided by observed data but not restricted to the values of observed data.
Two approaches to choose the set $\frakM$ have been studied in some depth.
The moment-based approach considers distributions whose moments (such as mean and covariance) satisfy certain conditions \cite{scarf1958min,popescu2007robust,delage2010distributionally,zymler2013distributionally}.
The distance-based approach considers distributions that are close, in the sense of a chosen distance, to a nominal distribution $\bnu$, such as an empirical distribution.
Popular choices of the distances are $\phi$-divergences \cite{ben2013robust,bayraksan2015data}, which include Kullback-Leibler divergence \cite{Jiang2015Data-driven}, Burg entropy \cite{wang2016likelihood}, and total variation distance \cite{sun2015convergence} as special cases, Prokhorov metric \cite{erdougan2006ambiguous}, and Wasserstein metric \cite{esfahani2015data,gao2016distributionally}.

In some practical settings, the decision maker may be aware that the random variables exhibit some dependence structure, even if the parameter values that specify the dependence are not known.  For example, the decision maker may be aware of an approximate linear dependence, measured by Pearson's product-moment correlation coefficient \cite{galton1886regression,pearson1895note}, or some form of rank dependence, such as measured by Spearman's $\rho$ \cite{spearman1904proof} or Kendall's $\tau$ \cite{kendall1948rank}.
In this paper, we are interested in DRSO problems that take into account one of these types of dependence structure.

\subsection{Linear-correlationally robust stochastic optimization}

We first consider linear dependence.
For instance, the moment uncertainty set in \cite{delage2010distributionally} is defined by
\begin{equation}
\label{eqn:delage uncertainty set}
\begin{aligned}
\frakM \ \ := \ \ \Big\{\bmu \in \cP(\Xi) \; : \; & \E_{\bmu}[(\bxi - m_{0})^{\top} \Sigma_{0}^{-1} (\bxi - m_{0})] \ \leq \ \gamma_{0}, \\
& \E_{\bmu}[(\bxi - m_{0})(\bxi - m_{0})^{\top}] \ \preceq \ \Sigma_{0}\Big\}
\end{aligned}
\end{equation}
where $m_{0}$ is a specified ``center'' vector (such as a sample mean), $\gamma_{0} \geq 0$ can be viewed as the squared radius of the confidence region of the mean vector, and $\Sigma_{0} \succeq 0$ is often chosen to be the sample covariance matrix $\hSigma$ inflated by some constant $\gamma \geq 1$.
%The second constraint requires that the linear correlation structure of $\bmu$, expressed in terms of the centered second-order moment matrix, is close to the nominal $\Sigma_{0}$.
The resulting moment-based DRSO problem is given by 
\begin{equation}
\label{eqn:delage}
\begin{aligned}
\inf_{x \in X} \sup_{\bmu \in \cP(\Xi)} \Big\{\E_{\bmu}[\Psi(x, \bxi)] \; : \; & \E_{\bmu}[(\bxi - m_{0})^{\top} \Sigma_{0}^{-1} (\bxi - m_{0})] \ \leq \ \gamma_{0},\\
& \E_{\bmu}[(\bxi - m_{0})(\bxi - m_{0})^{\top}] \ \preceq \ \Sigma_{0}\Big\}
\end{aligned}
\end{equation}
Similar to other moment-based approaches, this approach is based on the curious assumption that certain conditions on the moments are known but that nothing else about the relevant distribution is known.
More often in applications, either one has data from repeated observations of the quantity $\bxi$, or one has no data, and in both cases the moment conditions do not describe exactly what is known about $\bxi$ or its distribution.
Moreover, such sets $\frakM$ often lead to unrealistic worst-case distributions which make the resulting decisions $x$ overly conservative \cite{wang2016likelihood,goh2010distributionally}.
Another characteristic of worst-case distributions in the moment uncertainty set in~\eqref{eqn:delage uncertainty set} is given below.

\begin{example}[Degeneracy of moment uncertainty set]
\label{eg:delage}
Suppose that $\gamma_{0} = 0$, and that $\Psi$ satisfies either of the following conditions:
\begin{enumerate}[label=(\roman*),leftmargin=2em]
\item
For given $x$, $\Psi(x,\cdot)$ is a concave function whose domain contains $m_{0}$.
\item
For given $x$, $\Psi(x,\cdot)$ is an indicator function of a set $A(x) \subset \Xi$ which contains $m_{0}$. 
\end{enumerate}
Then for the given $x$, $\bdelta_{m_{0}}$ is a worst-case distribution, independent of the value of $\Sigma_{0}$.
\qed
\end{example}
Therefore, under these conditions, the second-order moment constraint $\E_{\bmu}[(\bxi - m_{0})(\bxi - m_{0})^{\top}] \preceq \Sigma_{0}$ has no effect on the problem.
%This observation is intuitive. 
%In fact, when $\Psi$ is concave, Jensen's inequality implies that spreading the probability mass out of the mean vector cannot increase the objective value, and thus the Dirac measure $\bdelta_{m_{0}}$ is always optimal. 
%When $\Psi$ is an indicator function, $\bdelta_{m_{0}}$ is a feasible solution with objective value 1, and thus is optimal.
%It follows that the optimal value of a stochastic optimization problem \eqref{eqn:delage} with concave objective equals $\Psi(\hm)$, and the optimal value of the uncertainty quantification problem $\sup_{\bmu} \Pr_{\bmu}[\bxi \! \in \! A]$ equals 1. Hence, in some sense, the formulation \eqref{eqn:delage} is the \textit{most} conservative way of hedging against uncertainty in both situations above.
%The main reasons for this phenomenon is that formulation~\eqref{eqn:delage} only considers first and second order moment information, which allows too much freedom in the distribution.
%We note that a concave objective function $\Psi$ occurs, for example, when it represents the second-stage objective function of a two-stage stochastic linear program with objective uncertainty (see, e.g., Corollary 5.4.(i) in \cite{esfahani2015data}).
%We remark that when $\Psi$ is convex, Jensen's inequality implies that the worst-case distribution tends to spread out the probability mass as much as possible until the second-order moment constraint becomes tight, which intuitively explains Proposition 3 in \cite{delage2010distributionally}.

To overcome some of the drawbacks of formulation~\eqref{eqn:delage}, we consider sets $\frakM_{1}$ of distributions that combine the second-order moment constraint and the Wasserstein distance constraint, as follows:
\begin{equation}\label{eqn:frakM_{1}}
\begin{aligned}
\frakM_{1} \ & := \ \left\{\bmu \in \cP(\Xi) \ : \ W_{p}(\bmu,\bnu) \, \leq \, R_{0}, \; \E_{\bmu}[(\bxi - m_{0})(\bxi - m_{0})^{\top}] \, \preceq \, \Sigma_{0}\right\},
\end{aligned}
\end{equation}
where $R_{0} > 0$, $m_{0} \in \R^{K}$, $\Sigma_{0} \succeq 0$, $\bnu$ is some nominal distribution, and $W_{p}(\bmu,\bnu)$ is the Wasserstein metric of order $p \geq 1$ between $\bmu$ and $\bnu$, defined as
\[\begin{aligned}
&W_{p}^{p}(\bmu,\bnu) \\
 := & \min_{\bgamma \in \cP(\Xi^{2})} \bigg\{\int_{\Xi^{2}} d^{p}(\xi,\zeta) \bgamma(\xi,\zeta) \; : \; \bgamma \textnormal{ has marginal distributions } \bmu, \bnu\bigg\},
\end{aligned}
\]
where $d$ denotes a chosen metric on $\Xi$.
The choices of $m_{0}$ and $\Sigma_{0}$ are similar to the moment-based approach.
The set $\frakM_{1}$ contains all distributions that are close to the nominal distribution in terms of the Wasserstein metric and that satisfy the linear correlation structure expressed in terms of the centered second-order moment constraint.
It has been shown recently that DRSO with Wasserstein metrics has advantages over DRSO with $\phi$-divergences \cite{esfahani2015data,gao2016distributionally}.
%For example, it yields a solution that is immune to data perturbation and has nice out-of-sample performance, and it is expected that this approach is often less conservative.
% \begin{example}[Wasserstein uncertainty set does not control correlation structure]
%   Let $\Psi(\xi)=\psi(a^{\top}\xi)$ for some $a\in\R^{K}\setminus\{0\}$, some monotone increasing strictly convex function $\psi$ and all $\xi\in\Xi$, and let $\bnu=\frac{1}{N}\sum_{i=1}^{N} \hxi^{i}$. Suppose the subgradient of $\Psi$ attains its maximum among $\{\hxi^{i}\}_{i=1}^{N}$ uniquely at $\hxi^{N}$. Define $\zeta:=\argmax_{||\zeta||\leq 1} a^{\top}\zeta$. Then it follows that
%   \[
%     \bmu_m:= \frac{1}{N}\sum_{i=1}^{N-1} \hxi^{i} + (1-\frac{1}{m})\bdelta_{\hxi^{i}} + \frac{1}{m}\bdelta_{mNR_{0}\zeta},\ m=1,2,\cdots
%   \]
%   is a sequence of distributions converging to the optimal value of DRSO with Wasserstein metric.
%   If the underlying distribution has two fully linearly correlated components, say $\bxi_{1}=\bxi_{2}$ almost surely, then as long as $\zeta_{1}\neq\zeta_{2}$, $\bmu_m$ cannot be the underlying distribution, since its first and second components are not fully correlated. \qed
% \end{example}
% This example indicates that Wasserstein uncertainty may yield a worst-case distribution which is unlikely to happen. We expect that this phenomenon would have a large impact on the performance of the solution when several components of the underlying distribution exhibit high correlations.

Given $\frakM_{1}$, we define the following linear-correlationally robust stochastic optimization problem:
\begin{equation}
\label{problem:moment}
\inf_{x \in X} \sup_{\bmu \in \frakM_{1}} \E_{\bmu}[\Psi(x,\bxi)]. \tag{\textsf{Linear-CRSO}}
\end{equation}
In the first part of the paper we derive a tractable reformulation of this problem, and we examine its performance.

\subsection{Rank-correlationally robust stochastic optimization}
\label{sec:intro:copulaCRSO}

In the second part of the paper we consider a different dependence structure, motivated by notions of rank dependence.
First, we give its definition, and then we illustrate its modeling flexibility via some examples.
Given the nominal distribution $\bnu \in \cP(\Xi)$, let $F_{k} : \R \mapsto [0,1]$ denote its $k$-th marginal cumulative distribution function, $k = 1,\ldots,K$,
and let $F : \R^{K} \mapsto [0,1]^{K}$ be given by
\[
F(\zeta) \ \ := \ \ (F_{1}(\zeta_{1}),\ldots,F_{K}(\zeta_{K})).
\]
Note that $F$ is the vector of marginal cumulative distribution functions, and not the joint cumulative distribution function of $\bnu$.
Let $\sfd$ be a metric on $[0,1]^{K}$. We define a semimetric $\sfd_{F}$ on $\Xi$ by
\[
\sfd_{F}(\xi,\zeta) \ \ := \ \ \liminf_{\xi^{n} \to \xi,\ \zeta^{n} \to \zeta} \sfd\left(F(\xi^{n}),F(\zeta^{n})\right)
\]
Note that $\sfd_{F}$ is lower semicontinuous.
Let $q \in [1,\infty]$.
For any $\bmu,\bnu \in \cP(\Xi)$, we define
\begin{equation}
\label{eqn:sfW}
\begin{aligned}
& \sfW_{q}(\bmu,\bnu) \ \ := \ \ \\
& \begin{cases}
\displaystyle\bigg(\! \inf_{\bgamma \in \cP(\Xi^{2})} \! \Big\{\int_{\Xi^{2}} \sfd_{F}^{q}(\xi,\zeta) \bgamma(d\xi,d\zeta) \ : \ \bgamma \textnormal{ has marginals } \bmu,\bnu\Big\} \! \bigg)^{1/q}, & \mbox{if } 1 \leq q < \infty, \\
\displaystyle \inf_{\bgamma \in \cP(\Xi^{2})} \Big\{\esssup{\bgamma}{(\xi,\zeta) \in \Xi^{2}} \ \sfd_{F}(\xi,\zeta) \  : \ \bgamma \textnormal{ has marginals }\bmu,\bnu\Big\}, & \mbox{if } q = \infty.
\end{cases}
\end{aligned}
\end{equation}
This is a transport metric that generalizes the Wasserstein metric in which the transportation cost is given by $\sfd_{F}(\cdot,\cdot)$.

We consider the set $\frakM_{2}$ of probability distributions, where
\begin{equation}
\label{eqn:frakM_{2}}
\frakM_{2} \ \ := \ \ \Big\{\bmu \in \cP(\Xi) \; : \; W_{p}(\bmu,\bnu) \leq R_{0}, \ \sfW_{q}(\bmu,\bnu) \leq r_{0}\Big\},
\end{equation}
where $R_{0}, r_{0} > 0$.
The constraints suggest that $\bmu$ is close to $\bnu$ in the sense of both Wasserstein metric $W_{p}$ and transport metric $\sfW_{q}$.
Next we provide some examples to illustrate the meaning of different versions of $\sfW_{q}$.

\begin{example}[Empirical ranking]
\label{eg:ranking}
Suppose $\bnu = \frac{1}{N} \sum_{i=1}^{N} \bdelta_{\hxi^{i}}$ and $\bmu = \frac{1}{N} \sum_{i=1}^{N} \bdelta_{\xi^{i}}$.
In this case, the minimization problem involved in the definition of the transport metric $\sfW_{q}$ becomes an assignment problem.
Index the points $\hxi^{1},\ldots,\hxi^{N}$ and $\xi^{1},\ldots,\xi^{N}$ in such a way that an optimal assignment is given by
$(\hxi^{1},\xi^{1}),\ldots,(\hxi^{N},\xi^{N})$.
Thus the constraint $\sfW_{q}(\bmu,\bnu) \leq r_{0}$ is equivalent to
\begin{eqnarray}
\Big[\frac{1}{N} \sum_{i=1}^{N} \sfd_{F}^{q}\left(\xi^{i},\hxi^{i}\right)\Big]^{1/q} & \leq & r_{0} \nonumber \\
\Leftrightarrow \ \ \ \Big[\frac{1}{N} \sum_{i=1}^{N} \liminf_{\xi^{i,n} \to \xi^{i},\ \hxi^{i,n} \to \hxi^{i}} \sfd\left(F(\xi^{i,n}),F(\hxi^{i,n})\right)\Big]^{1/q} & \leq & r_{0}
\label{eqn:ranking}
\end{eqnarray}
For a set of $N$ numbers $\{\xi^{i,n}_{k}\}_{i=1}^{N}$, the mapping
\[
\xi^{i,n}_{k} \mapsto F_{k}(\xi^{i,n}_{k})
\]
is non-decreasing and assigns each $\xi^{i,n}_{k}$ a value in $\{0,1/N,2/N,\ldots,1\}$.
Thus the vector $\left(F_{k}(\xi^{1,n}_{k}),\ldots,F_{k}(\xi^{N,n}_{k})\right)$ can be viewed as rankings of $\{\xi^{i,n}_{k}\}_{i=1}^{N}$ that take values in $\{\frac{i}{N} \, : \, 0 \leq i \leq N\}$ (same ranking for different elements are allowed).
Therefore, the constraint $\sfW_{q}(\bmu,\bnu)\leq r_{0}$ controls the difference between the rankings of the components of points in the supports of $\bmu$ and $\bnu$.
The following examples consider various special cases that explain this in more detail.
\qed
\end{example}

\begin{example}[Spearman's footrule distance]
 In the above example, when $q = 1$ and $\sfd(u,v) = \|u-v\|_{1}$, the above constraint~\eqref{eqn:ranking} becomes
\begin{eqnarray}
\frac{1}{N} \sum_{i=1}^{N} \sum_{k=1}^{K} \liminf_{\xi^{i,n}_{k} \to \xi^{i}_{k},\ \hxi^{i,n}_{k} \to \hxi^{i}_{k}} \big|F_{k}(\xi^{i,n}_{k}) - F_{k}(\hxi^{i,n}_{k})\big| & \leq & r_{0} \nonumber \\
\Leftrightarrow \ \ \ \sum_{k=1}^{K} \sum_{i=1}^{N} \liminf_{\xi^{i,n}_{k} \to \xi^{i}_{k},\ \hxi^{i,n}_{k} \to \hxi^{i}_{k}} \big|N F_{k}(\xi^{i,n}_{k}) - N F_{k}(\hxi^{i,n}_{k})\big| & \leq & N^{2} r_{0}
\label{eqn:footrule}
\end{eqnarray}
Note that $\sum_{i=1}^{N} \big|N F_{k}(\xi^{i,n}_{k}) - N F_{k}(\hxi^{i,n}_{k})\big|$ is a measure of the distance between the rankings of $\left(F_{k}(\xi^{i,n}_{k}) : i=1,\ldots,N\right)$ and $\left(F_{k}(\hxi^{i,n}_{k}) : i=1,\ldots,N\right)$, and is called Spearman's footrule in \cite{diaconis1977spearman}.
\qed
\end{example}

\begin{example}[Comonotonicity]
Let $K = 2$.
Suppose that $\bnu = \frac{1}{N} \sum_{i=1}^{N} \bdelta_{\hxi^{i}}$ is comonotonic, that is, for any two observations $\hxi^{i} = (\hxi^{i}_{1},\hxi^{i}_{2})$ and $\hxi^{j} = (\hxi^{j}_{1},\hxi^{j}_{2})$, it holds that $\hxi^{i}_{1} \leq \hxi^{j}_{1}$ if and only if $\hxi^{i}_{2} \leq \hxi^{j}_{2}$.
Without loss of generality, assume that $\hxi^{i}$ are sorted in increasing order.
In addition, assume that they have different values component-wise, that is, $\hxi^{1}_{k} < \hxi^{2}_{k} < \cdots < \hxi^{N}_{k}$ for all $k$.
As before, consider $\bmu = \frac{1}{N} \sum_{i=1}^{N} \bdelta_{\xi^{i}}$.
Let $q = 1$.
Then \eqref{eqn:footrule} is equivalent to
\[
\sum_{k=1}^{K} \sum_{i=1}^{N} \min\left\{\liminf_{\xi^{i,n}_{k} \to \xi^{i}_{k}} |N F_{k}(\xi^{i,n}_{k}) - (i-1)|, \, \liminf_{\xi^{i,n}_{k} \to \xi^{i}_{k}} |N F_{k}(\xi^{i,n}_{k}) - i)|\right\} \ \ \leq \ \ N^{2} r_{0},
\]
which is an approximate measure of the deviation of $\bmu$ from comonotonicity.
%In particular, $\sfW_{q}(\bmu,\bnu)=0$ means $\bmu$ is also comonotonic.
\qed
\end{example}

\begin{example}[Box uncertainty set for each individual data]
\label{eg:box}
Let $q = \infty$. Then constraint~\eqref{eqn:ranking} becomes
\[\begin{aligned}
\liminf_{\xi^{i,n} \to \xi^{i},\ \hxi^{i,n} \to \hxi^{i}} \sfd_{F}\big(\left(F_{1}(\xi^{i,n}_{1}),\ldots,F_{K}(\xi^{i,n}_{K})\right), \left(F_{1}(\hxi^{i,n}_{1}),\ldots,F_{K}(\hxi^{i,n}_{K})\right)\big) \ \ \leq \ \ r_{0}, \\
\quad \forall 1 \leq i \leq N.
\end{aligned}
\]
Thus each individual data point is constrained to be in a certain region.
In particular, when $\sfd(u,v) = ||u - v||_{\infty}$, then the above constraint becomes
\[
\begin{aligned}
\xi^{i}_{k} \ \leq \ \hxi^{j}_{k}, \quad \forall \ i, j \textnormal{ such that } F^{-}_{k}(\hxi^{j}_{k}) - F_{k}(\hxi^{i}_{k}) \geq r_{0}, \\
\xi^{i}_{k} \ \geq \ \hxi^{j}_{k}, \quad \forall \ i, j \textnormal{ such that } F^{-}_{k}(\hxi^{i}_{k}) - F_{k}(\hxi^{j}_{k}) \geq r_{0},
\end{aligned}
\]
where $F^{-}_{k}(\hxi_{k}) := \lim_{\xi_{k} \uparrow \hxi_{k}} F_{k}(\xi_{k})$ denotes the left limit of $F_{k}$ at $\hxi_{k}$.
Thus, each $\xi^{i}_{k}$ is constrained in some interval containing $\hxi^{i}_{k}$.
In particular, if $r_{0} = 0$, then $\xi^{i}_{k}$ belongs to $[\hxi^{i-}_{k},\hxi^{i+}_{k}]$, where $\hxi^{i-}_{k}$ (resp.~$\hxi^{i+}_{k}$) is the largest (resp.~smallest) data value among $\{\hxi^{i}_{k}\}_{i=1}^{N}$ that is strictly smaller (resp.~greater) than $\hxi^{i}_{k}$.
\qed
\end{example}

\begin{example}[Copula]
\label{eg:copula}
Let $\bmu$ be any continuous distribution on $\R^{K}$ with cumulative distribution function $H$. Sklar's theorem \cite{sklar1959fonctions} states that there exists a joint distribution $\bC^{\bmu}$ on $[0,1]^{K}$ with uniform marginals on $[0,1]$, such that $H$ can be expressed in terms of its marginal cumulative distribution functions $F^{\bmu}_{k}$ and $\bC^{\bmu}$ as follows:
\begin{equation}
\label{eqn:sklar}
H(\xi_{1},\ldots,\xi_{K}) \ \ = \ \ \bC^{\bmu}\left(F^{\bmu}_{1}(\xi_{1}),\ldots,F^{\bmu}_{K}(\xi_{K})\right), \quad \forall \ \xi \in \R^{K}.
\end{equation}
Such $\bC^{\bmu}$ is called a \textit{copula}.
Suppose that $F^{\bmu}_{k} = F_{k}$ for all $k$.
Then using change-of-variables, $\sfW_{q}(\bmu,\bnu)$ can be written as
\[
\begin{aligned}
&\sfW_{q}(\bmu,\bnu) \\
& \ \ = \ \ \bigg(\min_{\bgamma \in \cP(\Xi^{2})} \Big\{\int_{\Xi^{2}} \sfd^{q}(F(\xi),F(\zeta) \bgamma(d\xi,d\zeta) \ : \ \bgamma \textnormal{ has marginals } \bmu,\bnu\Big\}\bigg)^{1/q} \\
& \ \ = \ \ \bigg(\min_{\bpi \in \cP([0,1]^{2K})}  \Big\{\int_{[0,1]^{2K}} \sfd^{q}(u,v) \bpi(du,dv) \ : \ \bpi\textnormal{ has marginals } \bC^{\bmu},\bC^{\bnu}\Big\}\bigg)^{1/q}.
\end{aligned}
\]
% When $\bmu$ is close to $\bnu$ but has different marginals than $\bnu$, \eqref{eqn:sfW} can be viewed an approximation of the Wasserstein distance between copulas of $\bmu$ and $\bnu$.
Therefore, the constraint $\sfW_{q}(\bmu,\bnu) \leq r_{0}$ suggests that the copulas of $\bmu$ and $\bnu$ are close to each other in terms of Wasserstein distance.
\qed
\end{example}

Based on these examples, we observe that for data-driven problems in which the distribution is finite-supported, the constraint $\sfW_{q}(\bmu,\bnu) \leq r_{0}$ controls the difference between rankings of the distributions, and more generally, it controls the difference between copulas of the distributions.
The corresponding decision problem is given by
\begin{equation}
\label{problem:copula}
\inf_{x \in X} \sup_{\bmu \in \frakM_{2}} \E_{\bmu}[\Psi(x,\bxi)].
\tag{\textsf{Rank-CRSO}}
\end{equation}
We study this problem in the second part of the paper.

\subsection{Related literature}

The study of distributionally robust stochastic optimization can be traced back to \cite{vzavckova1966minimax,dupavcova1987minimax}.
%Recently it regained attention in the operations research literature, and sometimes is called data-driven stochastic optimization or ambiguous stochastic optimization.
We have already mentioned several papers using moment-based and distance-based approaches.
In addition, it is known that DRSO problems with $\frakM$ given by a $\phi$-divergence ball can easily incorporate moment constraints, simply by adding the constraints to the reformulation of the problem \cite{wang2016likelihood}.
However, for DRSO problems with $\frakM$ given by a Wasserstein ball, it has not been shown before whether moment constraints can be incorporated without losing tractability.
In addition, since the Wasserstein metric includes the total variation distance as a special case (see, e.g., Remark 4 in \cite{gao2016distributionally}), our result for problem \eqref{problem:moment} implies the result in \cite{Jiang2015Data-driven}, in which $\frakM$ is given by a total variation distance constraint and a moment constraint.
% Copula theory has be widely used in finance and statistics. However, it is not commonly used in stochastic optimization. \cite{henrion2011convexity,cheng2015chance} studied chance-constrained programming using copulas.
To the best of our knowledge, our paper is the first to study the effect of dependence structure (linear dependence and rank dependence) on stochastic optimization when only inexact information on the dependence structure is available.
A related topic is to study the worst-case performance of stochastic optimization problems when marginal distributions are known and no information on the dependence structure is available, see \cite{agrawal2012price,doan2012complexity,doan2015robustness}.

\subsection{Our contribution}
We summarize our results as follows:
\begin{itemize}
\item
Motivated by the poor performance of moment-based DRSO problems, we propose a new formulation~\eqref{problem:moment}, that takes into account the distribution's second-order moment information as well as Wasserstein distance to the nominal distribution.
Based on a constructive proof, we derive a tractable dual reformulation of this problem in Theorem~\ref{thm:moment}. 
\item
We investigate how the degree of correlation affects the performance of three DRSO approaches: DRSO with $\frakM$ given by a Wasserstein ball, DRSO with $\frakM$ given by moment constraints, and DRSO with $\frakM$ given by both a Wasserstein ball and moment constraints~\eqref{problem:moment}.
Numerical results on a portfolio optimization problem indicate that the new formulation outperforms the others in all (low-, medium-, high-) correlation regimes.
%Our duality results also provide a theoretical justification for an existing empirical evidence of the performance of LASSO and ridge regression in the presence of collinearity.
\item
We also propose a new formulation~\eqref{problem:copula} which, with appropriately chosen parameters, controls different dependence structures, including Spearman's footrule distance between empirical rankings, comonotonicity, box uncertainty for individual data points, and in general Wasserstein distance between copulas.
We also derive the dual reformulation for~\eqref{problem:copula} in Theorems~\ref{thm:copula} and \ref{thm:q=infty}.
% and justify the usefulness of the new formulation by comparing it with \eqref{problem:moment} and other choices of distance.
\end{itemize}

\section{Linear-correlationally robust stochastic optimization}
\label{sec:linearCRSO}

 In this section, we study linear correlationally robust stochastic optimization problem \eqref{problem:moment}.
We will derive a convex programming dual reformulation of \eqref{problem:moment} in Section \ref{sec:linearCRSO_{d}ual}, and apply it to a portfolio optimization problem in \ref{sec:portfolio}.

Since we focus only on the inner maximization problem, we suppress $x$ in $\Psi(x,\xi)$.
Throughout this section, we assume $1\leq p<\infty$, $\Psi$ is upper semi-continuous, and satisfies the growth rate condition
\[
\kappa \ \ := \ \ \limsup_{d(\xi,\zeta_{0})\to\infty}\frac{\max\big(\Psi(\xi)-\Psi(\zeta_{0}),0\big)}{d^{p}(\xi,\zeta_{0})} < \infty,
\]
for some $\zeta_{0} \in \Xi$.

\subsection{Dual reformulation of \eqref{problem:moment}}
\label{sec:linearCRSO_{d}ual}

The main result is given in the following Theorem \ref{thm:moment}.

\begin{theorem}\label{thm:moment}
Assume $\E_{\bnu}[(\bxi - m_{0})(\bxi - m_{0})^{\top}] \prec \Sigma_{0}$. Then~\eqref{problem:moment} has a strong dual problem
\[
\min_{\substack{\lambda\geq0\\\Lambda\succeq 0}}\left\{ \lambda R_{0}^{p} \! +\! \langle \Lambda,\Sigma_{0} \rangle \! +\! \int_\Xi \sup_{\xi\in\Xi} \big[\Psi(\xi) \! -\! \lambda d^{p}(\xi,\zeta) \! -\! (\xi-m_{0})^{\top} \Lambda(\xi-m_{0})\big]\bnu(d\zeta)\right\}.
\]
\end{theorem}

In the dual problem, $\lambda$ and $\Lambda$ are the Lagrangian multiplier of primal Wasserstein constraint and second-order moment constraint.
  The dual objective is a convex function of $\lambda$ and $\Lambda$.
  The measurability of the integrand is guaranteed by Lemma \ref{lemma:random_lsc} below. Denote by $(\Xi,\mathscr{B}_{\bnu}(\Xi),\bnu)$ the completion of measure space $(\Xi,\mathscr{B}(\Xi),\bnu)$ (see, e.g., Lemma 1.25 in \cite{kallenberg2006foundations}). A function $f:\R^m\times\Xi\to\bar{\R}$ is called a \textit{normal integrand} \cite{rockafellar2009variational}, if the associated epigraphical multifunction $\zeta\mapsto\mathrm{epi}\ f(\cdot,\zeta)$ is closed valued and measurable.

  \begin{lemma}\label{lemma:random_lsc}
    The function $\Phi:\R\times\R^{K\times K}\times\Xi$ defined by
    \[
      \Phi(\lambda,\Lambda,\zeta):=\sup_{\xi\in\Xi}[\Psi(\xi) - (\xi-m_{0})^{\top}\Lambda (\xi-m_{0}) - \lambda d^{p}(\xi,\zeta)]
    \]
    is a normal integrand with respect to $\mathscr{B}(\R)\otimes\mathscr{B}(\R^{K\times K})\otimes\mathscr{B}_{\bnu}(\Xi)$.
  \end{lemma}
  \begin{proof}
    Define a function $g:\Xi\times\R\times\R^{K\times K}\times\Xi\to\bar{\R}$ by
    \[
      g(\xi,\lambda,\Lambda,\zeta)=\Psi(\xi)-(\xi-m_{0})^{\top}\Lambda(\xi-m_{0})-\lambda d^{p}(\xi,\zeta).
    \]
    Then for every $\zeta\in\Xi$, $-g(\cdot,\cdot,\cdot,\zeta)$ is lower semi-continuous, thus $g$ is $\mathscr{B}(\Xi)\otimes\mathscr{B}(\R)\otimes\mathscr{B}(\R^{K\times K})\otimes\mathscr{B}_{\bnu}(\Xi)$-measurable. Hence by joint measurability criterion (see, e.g., Corollary 14.34 in \cite{rockafellar2009variational}), $g$ is a normal integrand, thereby the function $\Phi$ is also a normal integrand (Theorem 7.38 in \cite{shapiro2009lectures}). \qed
  \end{proof}

  \begin{proof}[Proof of Theorem \ref{thm:moment}] We divide the proof into four steps.
    \begin{enumerate}[label=\bf{Step} \arabic*.\ ,align=left, leftmargin=0pt,  listparindent=\parindent, labelwidth=0pt, itemindent=!,itemsep=1em]      
      \item We first show weak duality. Observe that for any random vector $(\bxi,\bzeta)$ with joint distribution $\bgamma\in\cP(\Xi^{2})$ and marginals $\bmu,\bnu\in\cP$, by property of conditional expectation, it holds that
        \[
          \int_\Xi \Psi(\xi)\bmu(d\xi) = \int_{\Xi^{2}} \Psi(\xi) \bgamma(d\xi,d\zeta) = \int_{\Xi^{2}} \Psi(\xi) \bgamma_{\zeta}(d\xi) \bnu(d\zeta),
        \]
        where $\bgamma_\zeta$ represents the conditional distribution of $\bxi$ given $\bzeta=\zeta$. Thus we can write problem \eqref{problem:moment} as
        \[
          \begin{aligned}
            & \sup_{\bmu\in\frakM_{1}} \int_\Xi \Psi(\xi) \bmu(d\xi)\\
           = & \sup_{\{\bgamma_\zeta\}_\zeta\subset\cP(\Xi)} \bigg\{\int_{\Xi^{2}} \Psi(\xi) \bgamma_{\zeta}(d\xi) \bnu(d\zeta): \int_{\Xi^{2}} d^{p}(\xi,\zeta) \bgamma_\zeta(d\xi)\bnu(d\zeta) \leq R_{0}^{p}, \\
           & \hspace{120pt}  \int_{\Xi^{2}} (\xi-m_{0})(\xi-m_{0})^{\top} \bgamma_\zeta(d\xi)\bnu(d\zeta) \preceq \Sigma_{0}\bigg\}.
           \end{aligned}
        \]
        The Lagrangian weak duality yields that 
        \[\begin{aligned}
            & \sup_{\bmu\in\frakM_{1}} \int_\Xi \Psi(\xi) \bmu(d\xi)\\
           \leq & \inf_{\lambda\geq0,\Lambda\succeq 0} \bigg\{ \lambda R_{0}^{p} + \langle \Lambda,\Sigma_{0}\rangle + \\
           & \ \sup_{\{\bgamma_\zeta\}_\zeta\subset\cP(\Xi)} \! \Big\{\! \int_{\Xi^{2}} \! \big[\Psi(\xi) -(\xi-m_{0})^{\top} \Lambda(\xi-m_{0}) -\lambda d^{p}(\xi,\zeta)\big] \bgamma_{\zeta}(d\xi) \bnu(d\zeta) \Big\} \bigg\}\\
           \leq & \inf_{\lambda\geq0,\Lambda\succeq 0} \bigg\{ \lambda R_{0}^{p} + \langle \Lambda,\Sigma_{0}\rangle \\
            & \hspace{63pt} +  \int_\Xi \sup_{\xi\in\Xi} [\Psi(\xi) -(\xi-m_{0})^{\top} \Lambda(\xi-m_{0}) -\lambda d^{p}(\xi,\zeta)] \bnu(d\zeta) \bigg\}.
          \end{aligned}
        \]

      \item We next show the existence of a dual minimizer. Let $h(\lambda,\Lambda)$ be the dual objective function. To begin with, let $v_{d}$ be the dual optimal value, and we claim that there exists $L>0$ such that
        \begin{equation}\label{eqn:bounded_lambda}
          v_{d} = \inf_{0\leq\lambda\leq L,\Lambda\succeq 0} h(\lambda,\Lambda).
        \end{equation}
        Indeed, the growth rate assumption on $\Psi$ implies that there exists $M>0$ such that $\Psi(\xi)-\Psi(\zeta_{0})\leq M d^{p}(\xi,\zeta_{0})$ for all $\xi\in\Xi$. By choosing $\lambda=M$ and $\Lambda=0$, we have that
        \[\begin{aligned}
          v_{d} & \leq MR_{0}^{p} + \int_\Xi \sup_{\xi\in\Xi}[\Psi(\zeta_{0}) + M (d^{p}(\xi,\zeta_{0}) - d^{p}(\xi,\zeta))] \bnu(d\zeta)\\
          & \leq MR_{0}^{p} + \Psi(\zeta_{0}) + M\int_\Xi d^{p}(\zeta_{0},\zeta) \bnu(d\zeta) < \infty.
          \end{aligned}
        \]
        On the other hand, for any feasible solution $(\lambda,\Lambda)$, it holds that
        \[
          \begin{aligned}
            h(\lambda,\Lambda)\geq & \lambda R_{0}^{p} + \langle \Lambda,\Sigma_{0} \rangle + \int_\Xi \big[\Psi(\zeta) - (\zeta-m_{0})^{\top} \Lambda(\zeta-m_{0})\big] \bnu(d\zeta)\\
            = & \lambda R_{0}^{p} + \Big\langle \Lambda, \Sigma_{0} - \int_\Xi (\zeta-m_{0})(\zeta-m_{0})^{\top} \bnu(d\zeta) \Big\rangle + \int_\Xi \Psi(\zeta) \bnu(d\zeta) \\
            \geq & \lambda  R_{0}^{p} + \int_\Xi \Psi(\zeta) \bnu(d\zeta),
          \end{aligned}
        \]
        which tends to $\infty$ as $\lambda\to\infty$.
        Hence the claim holds by choosing sufficiently large $L$.

        Now let $(\lambda^{(m)},\Lambda^{(m)})_m$ be a minimizing sequence of problem \eqref{eqn:bounded_lambda}. Since $(\lambda^{(m)})_m\subset[0,L]$, Bolzano-Weierstrass theorem implies that it has a convergent subsequence, whose limit is denoted by $\lambda_{\ast}$. Fixing $\lambda=\lambda_{\ast}$, the weak Lagrangian dual of the problem
        \begin{equation}\label{eqn:moment_{p}rimal_Gamma}
          \inf_{\Lambda\succeq 0} h(\lambda_{\ast},\Lambda)
        \end{equation}
        is given by
        \begin{equation}\label{eqn:moment_{d}ual_Gamma}\begin{aligned}
          &\inf_{\Lambda\succeq 0} \bigg\{ \lambda_{\ast} R_{0}^{p} + \langle \Lambda,\Sigma_{0}\rangle \\
          & \hspace{50pt} +  \int_\Xi \sup_{\xi\in\Xi} [\Psi(\xi) -(\xi-m_{0})^{\top} \Lambda(\xi-m_{0}) -\lambda_{\ast} d^{p}(\xi,\zeta)] \bnu(d\zeta) \bigg\}\\
          \geq& \inf_{\Lambda\succeq 0} \bigg\{ \lambda_{\ast} R_{0}^{p} + \langle \Lambda,\Sigma_{0}\rangle \\
          &\hspace{5pt} + \sup_{\{\bgamma_\zeta\}_\zeta\subset\cP(\Xi)} \int_{\Xi^{2}} [\Psi(\xi) -(\xi-m_{0})^{\top} \Lambda(\xi-m_{0}) -\lambda_{\ast} d^{p}(\xi,\zeta)] \bgamma_\zeta(d\xi)\bnu(d\zeta) \bigg\}\\
          \geq & \sup_{\{\bgamma_\zeta\}_\zeta\subset\cP(\Xi)} \inf_{\Lambda\succeq 0} \bigg\{\lambda_{\ast} R_{0}^{p} + \langle \Lambda,\Sigma_{0}\rangle \\
          &\hspace{45pt} + \int_{\Xi^{2}} [\Psi(\xi) -(\xi-m_{0})^{\top} \Lambda(\xi-m_{0}) -\lambda_{\ast} d^{p}(\xi,\zeta)] \bgamma_\zeta(d\xi)\bnu(d\zeta) \bigg\}\\
          =& \sup_{\bgamma\in\cP(\Xi^{2})} \bigg\{\lambda_{\ast} R_{0}^{p} + \int_\Xi [\Psi(\xi) -\lambda_{\ast} d^{p}(\xi,\zeta)] \bgamma(d\xi,d\zeta):\\
          &\hspace{80pt} \int_\Xi (\xi-m_{0})(\xi-m_{0})^{\top} \bgamma(d\xi,d\zeta) \preceq \Sigma_{0} \bigg\}.
          \end{aligned}
        \end{equation}
        Note that $\E_{\bnu}[(\bxi-m_{0})(\bxi-m_{0})^{\top}]\prec \Sigma_{0}$ , problem \eqref{eqn:moment_{d}ual_Gamma} satisfies the Slater condition, i.e., it is strictly feasible at $\bgamma_{0}$ defined by $\bgamma_{0}(A):=\bnu\{\zeta:(\zeta,\zeta)\in A\}$ for all Borel set $A\subset\Xi^{2}$, thus strong duality results for moment problem (cf. \cite{shapiro2001duality,delage2010distributionally}) implies the existence of a dual minimizer $\Lambda_{\ast}$ of problem \eqref{eqn:moment_{p}rimal_Gamma}. Therefore we have shown the existence of a dual minimizer $(\lambda_{\ast},\Lambda_{\ast})$.

      \item We then establish the first-order optimality condition of the dual problem. By Lemma \ref{lemma:random_lsc}, the function
      \[
        \Phi(\lambda,\Lambda,\zeta)= \sup_{\xi\in\Xi}[\Psi(\xi) - (\xi-m_{0})^{\top}\Lambda (\xi-m_{0}) - \lambda d^{p}(\xi,\zeta)]
      \]
      is a normal integrand. Moreover, for all $\zeta\in\Xi$, $\Phi(\cdot,\cdot,\zeta)$ is a convex function on $\R\times\R^{K\times K}$, and the growth rate condition implies that the set of maximizers is non-empty and compact for all $\lambda>\kappa$. Then by generalized Moreau-Rockafellar theorem (see, e.g., Theorem 7.47 in \cite{shapiro2009lectures}), for any $(\lambda,\Lambda)\in \mathrm{dom}\ h \cap ((\kappa,\infty)\times S_+^{K})$, it holds that
      \[
        \partial h(\lambda,\Lambda) = \left[\begin{array}{c}R_{0}^{p}\\\Sigma_{0}\end{array}\right] - \int_\Xi \partial \Phi(\lambda,\Lambda,\zeta) \bnu(d\zeta) + \mathcal{N}(\lambda,\Lambda),
      \]
      where $\mathcal{N}(\lambda,\Lambda)$ stands for the normal cone at $(\lambda,\Lambda)$ to the feasible region $\R_+\times S^{K}_+$. It follows from Theorem 2.4.18 in \cite{zalinescu2002convex} that for any $(\lambda,\Lambda)\in \mathrm{dom}\ h \cap ((\kappa,\infty)\times S_+^{K})$,
      \[\begin{aligned}
        &\partial \Phi(\lambda,\Lambda,\zeta) =  \conv\bigg\{\bigg[\begin{array}{c} d^{p}(\xi,\zeta)\\(\xi-m_{0})^{\top}\Lambda(\xi-m_{0})\end{array}\bigg]:\\ 
        & \hspace{90pt}\xi\in\argmax_{\xi\in\Xi}[\Psi(\xi) - (\xi-m_{0})^{\top}\Lambda (\xi-m_{0}) - \lambda d^{p}(\xi,\zeta)]\bigg\}.
        \end{aligned}
      \]
      Set
      \[
        T_\lambda(\zeta):= \argmax_{\xi\in\Xi} \ [\Psi(\xi) - (\xi-m_{0})^{\top}\Lambda_{\ast} (\xi-m_{0}) - \lambda d^{p}(\xi,\zeta)].
      \]
      The first-order optimality condition $0\in\partial h(\lambda_{\ast},\Lambda_{\ast})$ implies that there exists $\Sigma_{\ast}\in S^{K}_+$ with $\Sigma_{\ast}\preceq \Sigma_{0}$ and $\Lambda_{\ast}(\Sigma_{0}-\Sigma_{\ast})=0$, such that if $\lambda_{\ast}>\kappa$, it holds that
      \begin{equation}\label{eqn:moment_optimality}
        \left[\begin{array}{c} R_{0}^{p}\\ \Sigma_{\ast}\end{array}\right] \in \int_\Xi \conv\left\{\left[\begin{array}{c} d^{p}(\xi(\zeta),\zeta)\\ (\xi(\zeta)-m_{0})^{\top}\Lambda_{\ast}(\xi(\zeta)-m_{0})\end{array}\right]:\ \xi(\zeta)\in T_{\lambda_{\ast}}(\zeta)\right\}\bnu(d\zeta),
      \end{equation}
      and if $\lambda_{\ast}=\kappa$, for any $\lambda>\kappa$, there exists $R_\lambda \leq R_{0}$ such that
      \begin{equation}\label{eqn:moment_optimality0}
        \left[\begin{array}{c} R_\lambda^{p}\\ \Sigma_{\ast}\end{array}\right] \in \int_\Xi \conv\left\{\left[\begin{array}{c} d^{p}(\xi(\zeta),\zeta)\\ (\xi(\zeta)-m_{0})^{\top}\Lambda_{\ast}(\xi(\zeta)-m_{0})\end{array}\right]:\ \xi(\zeta)\in T_{\lambda}(\zeta)\right\}\bnu(d\zeta).
      \end{equation}

      \item Finally, we construct a primal (approximate) optimal solution.
      Let us first consider the case $\lambda_{\ast}>\kappa$. Similar to the argument in Step 4 of proof for Theorem \ref{thm:copula}, \eqref{eqn:moment_optimality} suggests that there exists a probability kernel $\{\bgamma^\ast_\zeta\}_{\zeta\in\Xi}$ such that each $\bgamma^\ast_\zeta$ is a probability distribution on $T_{\lambda_{\ast}}(\zeta)$, and satisfies
      \begin{subequations}
        \begin{align}
          \int_{\Xi^{2}} d^{p}(\xi,\zeta) \bgamma_\zeta^\ast(d\xi)\bnu(d\zeta)&=R_{0}^{p},\label{eqn:moment_{k}kt-a}\\
          \int_{\Xi^{2}} (\xi-m_{0})(\xi-m_{0})^{\top} \bgamma_\zeta^\ast(d\xi)\bnu(d\zeta)&=\Sigma_{\ast},\label{eqn:moment_{k}kt-b}\\
          \Sigma_{\ast} &\preceq \Sigma_{0},\label{eqn:moment_{k}kt-c}\\
          \Lambda_{\ast}(\Sigma_{0}-\Sigma_{\ast})&=0.\label{eqn:moment_{k}kt-d}
        \end{align}
      \end{subequations}
      Now define a probability measure $\bmu_{\ast}$ by
      \[
        \bmu_{\ast}(A):= \int_\Xi \bgamma^\ast_\zeta(A) \bnu(d\zeta),\ \forall A \in \mathscr{B}(\Xi).
      \]
      Then $\bmu_{\ast}$ is a primal feasible solution due to \eqref{eqn:moment_{k}kt-a}. In addition,
      \[\begin{aligned}
        &\int_\Xi \Psi(\xi) \bmu_{\ast}(d\xi)\\
         = & \int_{\Xi^{2}} \Psi(\xi)\bgamma_\zeta^\ast(d\xi)\bnu(d\zeta)\\
        = & \! \int_{\Xi^{2}} \! [\Psi(\xi)\!-\! (\xi-m_{0})^{\top} \Lambda_{\ast}(\xi-m_{0}) \!-\! \lambda_{\ast} d^{p}(\xi,\zeta)] \bgamma_\zeta^\ast(d\xi)\bnu(d\zeta) \! + \lambda_{\ast} R_{0}^{p} \! + \! \langle \Lambda_{\ast}, \Sigma_{\ast} \rangle \\
        = & \int_\Xi\max_{\xi\in\Xi}[\Psi(\xi)-(\xi-m_{0})^{\top} \Lambda_{\ast}(\xi-m_{0}) -\lambda_{\ast} d^{p}(\xi,\zeta)] \bnu(d\zeta) + \lambda_{\ast} R_{0}^{p} + \langle \Lambda_{\ast}, \Sigma_{0} \rangle\\
        \geq & v_{d},
        \end{aligned}
      \]
      where the second and the third equalities follows from \eqref{eqn:moment_{k}kt-a}-\eqref{eqn:moment_{k}kt-d}. 

      We then consider $\lambda_{\ast}=\kappa$. \eqref{eqn:moment_optimality0} suggests that for any $\lambda>\kappa$, there exists a probability kernel $\{\bgamma^\lambda_\zeta\}_{\zeta\in\Xi}$ such that each $\bgamma^\lambda_\zeta$ is a probability distribution on $T_{\lambda}(\zeta)$, and satisfies
      \begin{subequations}
        \begin{align}
          \int_{\Xi^{2}} d^{p}(\xi,\zeta) \bgamma_\zeta^\lambda(d\xi)\bnu(d\zeta)&=R_\lambda^{p},\label{eqn:moment_{k}kt-a0}\\
          \int_{\Xi^{2}} (\xi-m_{0})(\xi-m_{0})^{\top} \bgamma_\zeta^\ast(d\xi)\bnu(d\zeta)&=\Sigma_{\ast},\label{eqn:moment_{k}kt-b0}\\
          \Sigma_{\ast} &\preceq \Sigma_{0},\label{eqn:moment_{k}kt-c0}\\
          \Lambda_{\ast}(\Sigma_{0}-\Sigma_{\ast})&=0.\label{eqn:moment_{k}kt-d0}
        \end{align}
      \end{subequations}
      Now define a sequence of probabilty measures $\{\bmu^\lambda\}$ by
      \[
        \bmu^\lambda(A):= \int_\Xi \bgamma^\lambda_\zeta(A) \bnu(d\zeta),\ \forall A \in \mathscr{B}(\Xi).
      \]
      Then $\bmu^\lambda$ is a primal feasible solution due to \eqref{eqn:moment_{k}kt-a0}. In addition, from \eqref{eqn:moment_{k}kt-a0}-\eqref{eqn:moment_{k}kt-d0}
      \[\begin{aligned}
        &\int_\Xi \Psi(\xi) \bmu^\lambda(d\xi)\\
         = & \int_{\Xi^{2}} \Psi(\xi)\bgamma_\zeta^\lambda(d\xi)\bnu(d\zeta)\\
        = & \int_{\Xi^{2}}\! [\Psi(\xi) \!-\! (\xi-m_{0})^{\top} \Lambda_{\ast}(\xi-m_{0}) \!-\! \lambda d^{p}(\xi,\zeta)] \bgamma_\zeta^\lambda(d\xi)\bnu(d\zeta)  +  \lambda R_{0}^{p} + \langle \Lambda_{\ast}, \Sigma_{\ast} \rangle \\
        = & \int_\Xi\max_{\xi\in\Xi}[\Psi(\xi)-(\xi-m_{0})^{\top} \Lambda_{\ast}(\xi-m_{0}) -\lambda d^{p}(\xi,\zeta)] \bnu(d\zeta) + \lambda R_{0}^{p} + \langle \Lambda_{\ast}, \Sigma_{0} \rangle,
        \end{aligned}
      \]
      which goes to $h(\kappa,\Lambda_{\ast})$ as $\lambda\to\kappa$. Therefore, combined with Step 1, we have shown the strong duality. \qed
    \end{enumerate}
  \end{proof}

  \begin{remark}\label{rmk:concave}
    When $\Psi$ is a concave function and $\Xi$ is convex, $T_{\lambda_{\ast}}(\zeta)$ and $T_{\lambda}(\zeta)$ are convex sets, and thus the convex combination in \eqref{eqn:moment_optimality}\eqref{eqn:moment_optimality0} belong to $T_{\lambda_{\ast}}(\zeta)$ and $T_{\lambda}(\zeta)$ respectively. Thus we can replace the convex hull by a single point in $T_{\lambda_{\ast}}(\zeta)$ or $T_{\lambda}(\zeta)$. It follows that when $\Psi$ is concave and $\bnu=\frac{1}{N}\sum_{i=1}^{N} \bdelta_{\hxi^{i}}$, it is sufficient to restrict $\frakM_{1}$ on its subset $\frakM_{1}\cap\big\{\frac{1}{N}\sum_{i=1}^{N} \bdelta_{\xi^{i}}: \{\xi^{i}\}_{i=1}^{N}\subset\Xi\big\}$, the set of distributions in $\frakM_{1}$ that are supported on at most $N$ points:
    \[
      \max_{\xi^{i}}\bigg\{ \frac{1}{N}\sum_{i=1}^{N} \Psi(\xi^{i}):\ 
       \frac{1}{N}\sum_{i=1}^{N} d^{p}(\xi^{i},\hxi^{i})\leq R_{0}^{p},\  \frac{1}{N}\sum_{i=1}^{N} (\xi^{i}-m_{0})(\xi^{i}-m_{0})^{\top} \preceq \Sigma_{0}\bigg\}.
    \]
    \qed
  \end{remark}

  For piecewise linear convex objective function $\displaystyle \Psi(\xi)=\max_{1\leq j\leq J} a_j^{\top} \xi+b_j$ and empirical nominal distribution, we have the following result.
  \begin{corollary}\label{cor:moment}
    Suppose $\displaystyle \Psi(\xi)=\max_{1\leq j\leq J} a_j^{\top} \xi+b_j$, $\bnu = \frac{1}{N}\sum_{i=1}^{N} \bdelta_{\hxi^{i}}$, $p=1$ and $d(\xi,\zeta)=||\xi-\zeta||$. Then the DRSO problem admits a semi-definite program reformulation:
    \[
      \begin{aligned}
      \min_{\substack{y^{i}\in\R\\\lambda\geq0,\Lambda\succeq 0}} & \ \  \lambda R_{0} + \langle \Lambda,\Sigma_{0} \rangle + \frac{1}{N}\sum_{i=1}^{N} y_{i}\\
      s.t. \ \  &
      \left[\begin{array}{cc} \Lambda & -\frac{1}{2} a_j + \frac{1}{2}\zeta^{i} - \Lambda m_{0} \\ (-\frac{1}{2} a_j + \frac{1}{2}\zeta^{i} - \Lambda m_{0})^{\top} & - b_j-{\zeta^{i}}^{\top}\hxi^{i} + m_{0}^{\top}\Lambda m_{0}+y_{i} \end{array}\right]\!\succeq\!0,\\
      &\hspace{140pt} \forall 1\leq j\leq J,\ \forall 1\leq i\leq N,\\
      &  ||\zeta^{i}||_{\ast} \leq \lambda,\quad \forall 1\leq i\leq N,
    \end{aligned}
    \]
    where $||\cdot||_{\ast}$ denotes the dual norm of $||\cdot||$. 
  \end{corollary}
  \begin{proof}
      Observe that $||\xi-\hxi^{i}||=\sup_{||\zeta||_{\ast}\leq1} \zeta^{\top}(\xi-\hxi^{i})$, then by convex programming duality, we have that
      \[\begin{aligned}
        &\max_{\xi\in\Xi} \Big\{a_j^{\top}\xi + b_j - \lambda||\xi-\hxi^{i}|| - (\xi-m_{0})^{\top}\Lambda(\xi-m_{0})\Big\}\\
        =& \max_{\xi\in\Xi}  \inf_{||\zeta||_{\ast}\leq 1} \Big\{ a_j^{\top}\xi + b_j - \lambda\zeta^{\top}(\xi-\hxi^{i}) - (\xi-m_{0})^{\top}\Lambda(\xi-m_{0})\Big\}\\
        =&  \inf_{||\zeta||_{\ast}\leq 1} \max_{\xi\in\Xi}  \Big\{  \sum_{j}\alpha_{ij} (a_j^{\top}\xi  + b_j) - \lambda\zeta^{\top}(\xi-\hxi^{i}) - (\xi - m_{0})^{\top}\Lambda(\xi - m_{0})\Big\}.
        \end{aligned}
      \]
      Hence the constraint
      \[
        y_{i} \geq \max_{\xi\in\Xi} \Big\{a_j^{\top}\xi + b_j - \lambda||\xi-\hxi^{i}|| - (\xi-m_{0})^{\top}\Lambda(\xi-m_{0})\Big\}
      \]
      can be written as
      \[
        \exists ||\zeta||_{\ast} \leq 1, \ s.t.\
        y_{i} \geq \max_{\xi\in\Xi} \Big\{a_j^{\top}\xi + b_j - \lambda\zeta^{\top}(\xi-\hxi^{i}) - (\xi-m_{0})^{\top}\Lambda(\xi-m_{0})\Big\},
      \]
      which is further equivalent to
      \[
        \exists ||\zeta||_{\ast}\leq 1,\  s.t.\ \left[\begin{array}{cc} \Lambda & -\frac{1}{2} a_j + \frac{1}{2}\lambda\zeta - \Lambda m_{0} \\ (-\frac{1}{2}a_j + \frac{1}{2}\lambda\zeta - \Lambda m_{0})^{\top}  & -b_j-\lambda\zeta^{\top}\hxi^{i} + m_{0}^{\top}\Lambda m_{0}+y_{i} \end{array}\right]\succeq0.
      \]
      Replacing $\lambda\zeta$ by $\zeta^{i}$ we obtain the result. \qed
    \end{proof}

  \subsection{Application in portfolio optimization}\label{sec:portfolio}
    In this section, we study the effect of degree of correlation on the performance of \eqref{problem:moment}, through a mean-risk portfolio optimization problem adapted from \cite{esfahani2015data}. In this problem, the random returns of $K$ risky assets is captured by a random vector $\bxi=(\bxi_{1},\ldots,\bxi_{k})$, the decision variable $x\in X=\{x\in\R_+^{K}:\ \sum_{k=1}^{K} x_{k}=1\}$ represents the portfolio weights without short-selling, and the goal is to minimize a weighted combination of the mean and conditional value-at-risk of the negative of portfolio return $-x^{\top}\bxi$:
    \[
      \min_{x\in X}\sup_{\bmu\in\frakM_{1}}\bigg\{ \E_{\bmu}[-x^{\top}\bxi] + c\cdot \mathrm{CVaR}^{\bmu}_{\alpha}(-x^{\top}\bxi) \bigg\},
    \]
    where $c$ is some constant. 
    It can be written as (see \cite{esfahani2015data} for a derivation)
    \begin{equation}\label{problem:portfolio}
      \inf_{x\in X,\tau\in\R}\sup_{\bmu\in\frakM_{1}}\bigg\{ \max_{1\leq j\leq J} a_jx^{\top}\bxi + b_j\tau\bigg\},
    \end{equation}
    where $J=2$, $a_{1}=-1$, $a_{2}=-1-c/\alpha$, $b_{1}=c$ and $b_{2}=c(1-1/\alpha)$.
    Suppose the nominal distribution $\bnu$ is chosen as the empirical distribution, then we can obtain a semi-definite programming reformulation using Corollary \ref{cor:moment}. In defining $\frakM_{1}$, we set $m_{0}$ to be the sample mean vector, and $\Sigma_{0}$ to be sample covariance matrix $\hSigma$ inflated by a factor $\gamma\geq1$.

    The parameters for the numerical experiments are given as follows. We set $K=10$, $c=10$, $\alpha=20\%$. We consider Wasserstein distance of order $p=1$ and $\Xi$ to be the Euclidean space $\R^{K}$.
    Assume $\bxi$ is joint Gaussian distributed, and each $\bxi_{k}$ has mean $0.03k$ and standard deviation $0.025k$, $k=1,\ldots,K$. The correlation between $\bxi_{k}$ and $\bxi_j$ is set to be $\rho^{|k-j|}$ for $\rho=0.5,0.6,0.7,0.8,0.9,0.99$.
    Note that in $0.9^5\simeq 0.69$ and $0.6^5\simeq0.08$, we can view $\rho=0.99,0.9$ as high correlation regime, $\rho=0.8,0.7$ as medium correlation regime, and $\rho=0.6,0.5$ as low correlation regime.
    We consider small dataset regime $N=40$, for which DRSO approach should be more suitable than SAA (sample average approximation) method. We run the simulation with 200 repetitions.

The tuning parameters (Wasserstein radius $R_{0}$ and inflation factor $\gamma$) are selected using hold-out cross validation.
    In each repetition, the $N$ samples are randomly partitioned into a training dataset (70\% of the data) and a validation dataset (the remaining 30\%).
    For different tuning parameters, we use the training dataset to solve problem \eqref{problem:portfolio} and use the validation dataset to estimate the out-of-sample performance of different parameter values and select the ones with the best performance.
    Then we resolve problem \eqref{problem:portfolio} with the best tuning parameters using all $N$ samples and obtain the optimal solutions for the three uncertainty sets. Finally we examine the out-of-sample performance of these solutions using an independent testing dataset consisting of $10^3$ samples.

    \begin{figure}[!t]
      \centering
      \includegraphics[width=\textwidth]{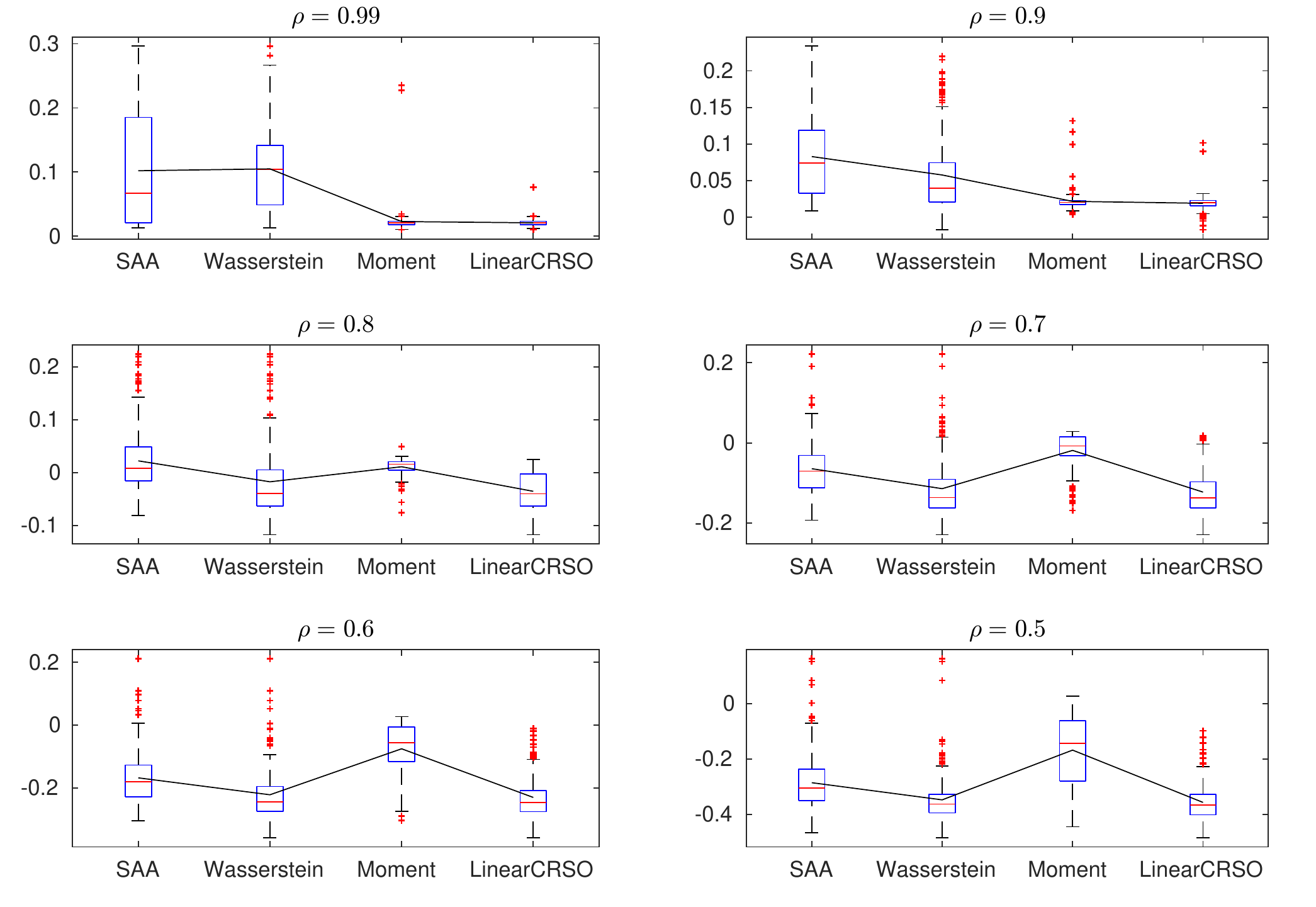}
      \caption{Out-of-sample performance under different degrees of correlation}\label{fig:portfolio}
    \end{figure}

    Figure \ref{fig:portfolio} shows the box plots of the optimal values in 200 repetitions of four different approaches: Sample Average Approximation, DRSO with Wasserstein uncertainty set, DRSO with moment uncertainty set and \eqref{problem:moment}, and the solid curves represent the average performance of these approaches.
    We observe that our new formulation \eqref{problem:moment} performs consistently the best in all regimes; DRSO with moment uncertainty set \eqref{eqn:delage} performs well in high-correlation regime; and DRSO with Wasserstein uncertainty set performs well in medium- and low-correlation regime.

    To provide a rough explanation of the results, we begin with describing the behavior of the worst-case distributions in three approaches.
    Observe that the objective of problem \eqref{problem:portfolio} is a weighted combination of mean and CVaR and is piecewise-linear convex in $\xi$. To achieve the inner maximization, the worst-case distributions of Wasserstein uncertainty set and \eqref{problem:moment} tend to perturb the data points with small returns (large value of $-x^{\top}\xi$) in the direction of minimizing $-x^{\top}\xi$. Meanwhile, for moment uncertainty set, according to Proposition 3 in \cite{delage2010distributionally}, the worst-case distribution also tends to spread out the probability mass towards this direction until the centered second-order moment constraint becomes tight.
    
    Now let us consider the situations in different correlation regimes.
    In high-correlation regime, the optimal portfolio under true underlying distribution puts almost all the weight on the first asset, which has the smallest weighted combination of mean and CVaR among all $K$ assets. 
    Since the expected returns are almost comonotonic, DRSO with moment uncertainty set becomes finding a worst-case distribution among all univariate distributions with given mean and variance.
    This is a relatively small set of distributions, and considering the underlying distribution is Gaussian, the solution yielding from moment uncertainty set is not overly conservative, and can effectively identify a portfolio close to the true optimal one.
    On the other hand, for Wasserstein uncertainty set, the worst-case distribution perturbs the data with small returns in the direction of minimizing $-x^{\top}\xi$. 
    Note that in the presence of high correlation, data points are concentrating on a subspace of $\R^{K}$ whose dimension is much less than $K$, and they can be easily perturbed out of this subspace. 
    Therefore the worst-case distribution may not lie in the same subspace as the empirical distribution.
    Hence the Wasserstein uncertainty set may hedge against some distributions unlikely to happen, which makes the decision over-conservative. 
    Indeed, the solution yielding from Wasserstein uncertainty set tends to equally allocate the weights to all assets.
    In medium- and low-correlation regime, the data points are not so concentrating, so perturbation towards any direction does not affect the correlation structure too much, hence Wasserstein constraint alone performs well, and almost as good as \eqref{problem:moment} in low-correlation regime. But moment uncertainty set is too conservative since it contains too many distributions. As a hybrid approach, \eqref{problem:moment} take advantages of Wasserstein and moment uncertainty sets.

\section{Rank-correlationally robust stochastic optimization}\label{sec:copulaCRSO}
  In this section, we study rank-correlationally robust stochastic optimization problem \eqref{problem:copula}. Dual reformulation of the problem is derived in \ref{sec:copulaCRSO_{d}ual}, and comparison of \eqref{problem:copula} with other approaches are in \ref{sec:compare_linear_copula} and \ref{sec:compare_wasserstein}.
  As in previous section, we suppress $x$ in $\Psi(x,\xi)$. Throughout this section, we assume $p\in[1,\infty)$, $q\in[1,\infty]$, $\Psi$ is upper semi-continuous, and satisfies the growth rate condition
  \[
    \limsup_{d(\xi,\zeta_{0})\to\infty}\frac{\max\big(\Psi(\xi)-\Psi(\zeta_{0}),0\big)}{d^{p}(\xi,\zeta_{0})} =0,
  \]
  for some $\zeta_{0}\in\Xi$.

  \subsection{Dual reformulation of \eqref{problem:copula}}\label{sec:copulaCRSO_{d}ual}
    We consider $q\in[1,\infty)$ in Theorem \ref{thm:copula} and $q=\infty$ in Theorem \ref{thm:q=infty}. 

    % Let $\frakF_{k}$ be the set of distribution functions from $[0,1]$ to $\Xi_{k}$. Set
    % \[
    %   F(\xi) := \big(F_{1}(\xi_{1}),\ldots,F_{k}(\xi_{k})\big),\quad \forall \xi\in\Xi.
    % \]
    % Note that $F$ is not the distribution function $H$ of $\bxi$ unless the components of $\bxi$ are mutually independent. 
    
    \begin{theorem}\label{thm:copula}
      Suppose $1\leq q<\infty$.
      \begin{enumerate}[label=(\roman*)]
        \item \label{itm:dual<infty} Problem \eqref{problem:copula} can be reformulated as
          \begin{equation}\label{eqn:dual<infty}
            %\sup_{F_{k}\in\frakF_{k},\forall k}  
            \min_{\alpha,\beta\geq0} \bigg\{ \alpha R_{0}^{p} + \beta r_{0}^{q} +  \int_{\Xi} \max_{\xi\in\Xi}  \big[\Psi(\xi) -  \alpha d^{p}(\xi,\zeta)  - \beta \sfd_{F}^{q}(\xi,\zeta)\big] \bnu(d\zeta)  \bigg\}.
          \end{equation}
        \item \label{itm:approx<infty} For data-driven problem, i.e., $\bnu=\frac{1}{N}\sum_{i=1}^{N} \bdelta_{\hxi^{i}}$, the following program
          \begin{equation}\label{eqn:data-driven<infty}
            \begin{aligned}
              \max_{\{\xi^{i}\}\subset\Xi} \bigg\{ & \frac{1}{N}\sum_{i=1}^{N} \Psi(\xi^{i}):\ \frac{1}{N}\sum_{i=1}^{N} d^{p}(\xi^{i},\hxi^{i})\leq R_{0}^{p},\  \frac{1}{N}\sum_{i=1}^{N} \sfd_{F}^{q}(\xi^{i},\hxi^{i}) \leq r_{0}^{q}\bigg\}
            \end{aligned}
          \end{equation}
          is a $(1-O(\frac{1}{N}))$-approximation of the inner maximization of \eqref{problem:copula}.
      \end{enumerate}
    \end{theorem}

    \begin{remark}
      In Section \ref{sec:intro:copulaCRSO}, we have showed various examples regrading different notions of distance by assuming the relevant distribution $\bmu$ is supported on (at most) $N$ points.
      Thanks to statement \ref{itm:approx<infty} (and also Theorem \ref{thm:q=infty}\ref{itm:data-driven} below), for data-driven problems, this assumption is almost not restrictive at all, since the worst-case distribution of the approximation problem \eqref{eqn:data-driven<infty} is supported on at most $N$ points. 
      \qed
    \end{remark}

    \begin{proof}[Proof of Theorem \ref{thm:copula}]
      We first prove \ref{itm:dual<infty}. The idea is similar to the proof of Theorem \ref{thm:moment}. The measurability of the integrand in the dual problem follows similarly to Lemma \ref{lemma:random_lsc}, so we omit the proof. Using Lagrangian weak duality we have weak duality
      \[
        \begin{aligned}
        &\sup_{\bmu\in\frakM_{2}} \bigg\{\int_{\Xi} \Psi(\xi) \bmu(d\xi)\bigg\}\\
        = & \sup_{\{\bgamma_\zeta\}_\zeta\subset\cP(\Xi)} \inf_{\alpha,\beta\geq0} \bigg\{\int_{\Xi^{2}} \Psi(\xi) \bgamma_\zeta(d\zeta)\bnu(d\zeta) + \alpha R_{0}^{p} - \alpha\int_{\Xi^{2}} d^{q}(\xi,\zeta) \bgamma_\zeta(d\xi)\bnu(d\zeta) \\
         & \hspace{170pt}  + \beta r_{0}^{q} - \beta \int_{\Xi^{2}} \sfd_{F}^{q}(\xi,\zeta) \bgamma_\zeta(d\xi)\bnu(d\zeta) \bigg\}\\
        \leq & \inf_{\alpha,\beta\geq0} \bigg\{ \alpha R_{0}^{p} + \beta r_{0}^{q} +  \int_{\Xi}\max_{\xi\in\Xi}  \big[\Psi(\xi) -  \alpha d^{p}(\xi,\zeta)  - \beta \sfd_{F}^{q}(\xi,\zeta)\big] \bnu(d\zeta) \bigg\}.
        \end{aligned}
      \]

      We next show that the dual problem admits a minimizer. Let $v_{d}$ be the optimal value of the dual problem. We claim that there exists $L>0$ such that
      \begin{equation}\label{eqn:dual_M}\begin{aligned}
        &v_{d}=\\
        &\inf_{0\leq \alpha,\beta\leq L}  \bigg\{\alpha R_{0}^{p}  +\beta r_{0}^{q}  +  \int_{\Xi}  \max_{\xi\in\Xi}  \big[\Psi(\xi) -  \alpha d^{p}(\xi,\zeta)  - \beta \sfd_{F}^{q}(\xi,\zeta)\big] \bnu(d\zeta)\bigg\}.
        \end{aligned}
      \end{equation}
      Indeed, according to the growth-rate assumption on $\Psi$, there exists $M>0$ and $\zeta_{0}\in\Xi$ such that $\Psi(\xi)-\Psi(\zeta_{0})\leq Md^{p}(\xi,\zeta_{0})$ for all $\xi\in\Xi$, thus by choosing $\alpha=M$ and $\beta=0$, we obtain that 
      \[\begin{aligned}
        v_{d} \leq & MR_{0}^{p} + \int_{\Xi} \max_{\xi\in\Xi} [\Psi(\xi) + M(d^{p}(\xi,\zeta_{0})-d^{p}(\xi,\zeta))]\bnu(d\zeta)\\
        \leq & MR_{0}^{p} + \Psi(\xi) + M\int_{\Xi} d^{p}(\zeta,\zeta_{0})\bnu(d\zeta)<\infty.
        \end{aligned}
      \]
      Meanwhile, 
      \[\begin{aligned}
          &\alpha R_{0}^{p} + \beta r_{0}^{q} +  \int_{\Xi}  \max_{\xi\in\Xi}  \big[\Psi(\xi) -  \alpha d^{p}(\xi,\zeta)  - \beta \sfd_{F}^{q}(\xi,\zeta)\big] \bnu(d\zeta)\\
          \geq & \alpha R_{0}^{p} + \beta r_{0}^{q} + \int_\Xi \Psi(\zeta) \bnu(d\zeta),
        \end{aligned}
      \]
      which tends to infinity as $\max(\alpha,\beta)\to\infty$. Hence the claim holds. Then by Bolzano-Weierstrass theorem there exists a minimizer.

      We then establish the first-order optimality condition of problem \eqref{eqn:dual_M} at a minimizer $({\alpha_{\ast}},\beta_{\ast})$. Let $h(\alpha,\beta)$ be the objective function and set
      \[
        \Phi(\alpha,\beta,\zeta) = \Psi(\xi) - \alpha d^{p}(\xi,\zeta) - \beta \sfd_{F}^{q}(\xi,\zeta).
      \]
      For any $(\alpha,\beta)\in \mathrm{dom}\ h$ we compute the differential of $h$ and $\Phi$ with respect to $\alpha,\beta$:
      \[
        \partial h(\alpha,\beta)= (R_{0}^{p},r_{0}^{q})^{\top} - \int_\Xi \partial_{\alpha,\beta} \Phi_n(\alpha,\beta,v) \bnu(d\zeta) + \mathcal{N}(\alpha,\beta),
      \]
      where $\mathcal{N}(\alpha,\beta)$ stands for the normal cone at $(\alpha,\beta)$ to the feasible region $\R_+^{2}$, and
      \[\begin{aligned}
        & \partial_{\alpha,\beta} \Phi(\alpha,\beta,\zeta)
         = \conv \Big\{\big( d^{p}(\xi,\zeta),\sfd_{F}^{q}(\xi,\zeta) \big)^{\top}: \ \xi\in T(\zeta)\Big\},
        \end{aligned}
      \]
      where
      \[
        T(\zeta):= \argmax_{\xi\in\Xi}\big[\Psi(\xi) - {\alpha_{\ast}} d^{p}(\xi,\zeta) - \beta_{\ast} \sfd_{F}^{q}(\xi,\zeta)\big].
      \]
      The first-order optimality condition $0\in\partial h({\alpha_{\ast}},\beta_{\ast},0)$ implies that there exists $0\leq R_{\ast} \leq R_{0}$ with ${\alpha_{\ast}}(R_{0}-R_{\ast})=0$ and $0\leq r_{\ast}\leq r_{0}$ with ${\beta_{\ast}}(r_{0}-r_{\ast})=0$, such that
      \begin{equation}\label{eqn:copula_optimality}
        \begin{aligned}
        (R^{p}_{\ast}, r^{q}_{\ast}) \in \int_\Xi \conv\Big\{\big(d^{p}(\xi(\zeta),\zeta),\sfd_{F}^{q}(\xi(\zeta),\zeta)\big): \xi(\zeta) \in T(\zeta)\Big\} \bnu(d\zeta).
        \end{aligned}
      \end{equation}
      
      Finally we construct a primal optimal solution. \eqref{eqn:copula_optimality} suggests that there is a probability kernel $\{\bgamma_\zeta^\ast\}_{\zeta\in\Xi}$ such that each $\bgamma_\zeta^\ast$ is a probability distribution on $T(\zeta)$ and satisfies 
      \begin{subequations}\label{eqn:copula_{k}kt}
        \begin{align}
          \int_{\Xi^{2}} d^{p}(\xi,\zeta) \bgamma_\zeta^\ast(d\xi)\bnu(d\zeta)&=R_{\ast}^{p},\label{eqn:copula_{k}kt-Rint}\\
          R_{\ast} &\leq R_{0},\label{eqn:copula_{k}kt-R}\\
          {\alpha_{\ast}}(R_{0}-R_{\ast})&=0,\label{eqn:copula_{k}kt-Rcomp}\\
          \int_{\Xi^{2}} \sfd_{F}^{q}(\xi,\zeta)\bgamma_\zeta^\ast(d\xi)\bnu(d\zeta) &=r_{\ast}^{q},\label{eqn:copula_{k}kt-rint}\\
          r_{\ast} &\leq r_{0},\label{eqn:copula_{k}kt-r}\\
          {\beta_{\ast}}(r_{0}-r_{\ast})&=0.\label{eqn:copula_{k}kt-rcomp}
        \end{align}
      \end{subequations}
      We can verify that the probability measure $\bmu_{\ast}$ by
      \[
        \bmu_{\ast}(A):= \int_{\Xi} \bgamma_\zeta^\ast(A) \bnu(d\zeta),\ \forall A \in \mathscr{B}(\Xi).
      \]
      is a primal feasible and optimal, and thus we have prove \ref{itm:dual<infty}.

      \vspace{11pt}
      Now we prove \ref{itm:approx<infty}. Let $(\alpha_{\ast},\beta_{\ast})$ be the optimal dual solution, then for data-driven problem, the dual optimal values equals
      \begin{equation}\label{eqn:dual_ast}
        \alpha_{\ast} R_{0}^{p} + \beta_{\ast} r_{0}^{q} + \frac{1}{N} \sum_{i=1}^{N} \max_{\xi\in\Xi}  \big[\Psi(\xi) -  \alpha_{\ast} d^{p}(\xi,\hxi^{i})  - \beta_{\ast} \sfd_{F}^{q}(\xi,\hxi^{i})\big],
      \end{equation}
      Caratheodory's theorem implies that for each $\hxi^{i}$, we can choose $\bgamma_{\hxi^{i}}^\ast$ defined by \eqref{eqn:copula_{k}kt} such that its support contains at most three points. Hence, the third term in problem \eqref{eqn:dual_ast} is equivalent to
      \[\begin{aligned}
        \max_{p^{ij}\geq0,\xi^{ij}\in\Xi} \bigg\{ &\frac{1}{N} \sum_{i=1}^{N} \sum_{j=1}^3 p^{ij} \big[ \Psi(\xi^{ij}) - \alpha_{\ast} d^{p}(\xi^{ij},\hxi^{i}) - \beta_{\ast} \sfd_{F}^{q}(\xi^{ij},\hxi^{i}) \big]: \\
        & \hspace{190pt}  \sum_{j}p^{ij} = 1,\ \forall i \bigg\},
        \end{aligned}
      \]
      which is also equivalent to
      \[
        \begin{aligned}
        \max_{p^{ij}\geq0,\xi^{ij}\in\Xi} \bigg\{ &\frac{1}{N} \sum_{i=1}^{N} \sum_{j=1}^3 p^{ij} \Psi(\xi^{ij}):\ \sum_{j}p^{ij} = 1,\ \forall i,\\
        &  \frac{1}{N}\sum_{i=1}^{N} \sum_{j=1}^3 p^{ij} \big[\alpha_{\ast} d^{p}(\xi^{ij},\hxi^{i}) + \beta_{\ast} \sfd_{F}^{q}(\xi^{ij},\hxi^{i})\big] \leq \alpha_{\ast} R_{0}^{p} + \beta_{\ast} r_{0}^{q} \bigg\},
        \end{aligned}
      \]
      Let us fix $\{\xi^{ij}\}_{ij}$ to be the optimal value $\{\xi_{\ast}^{ij}\}_{ij}$ in the program above and consider finding the best $\{p^{ij}\}_{ij}$. This is continuous knapsack problem which can be solved by a greedy algorithm. More specifically, those $(i,j)$ with large value of $\Psi(\xi^{ij})$ are preferably to be set to one, until the knapsack constraints is violated, when we may set such $p^{ij}$ to be a fractional value to make the knapsack constraint tight. Hence, there exists an optimal solution $\{p^{ij}_{\ast}\}_{ij}$ with at most two fractional points. Now we modify it by making the fractional values to be zero, and denote the modified solution by $\{\tilde{p}^{ij}\}_{ij}$. Then $\{\tilde{p}^{ij}\}_{ij}$ is also feasible, and yields a feasible distribution of \eqref{eqn:data-driven<infty}
      \[
        \frac{1}{N}\sum_{i=1}^{N} \Big[\sum_{j=1}^3 \tilde{p}^{ij}\bdelta_{\xi^{ij}_{\ast}} + (1-\sum_{j=1}^3\tilde{p}^{ij})\bdelta_{\hxi^{i}}\Big].
      \]
      The growth-rate assumption on $\Psi$ implies that $\xi_{\ast}^{ij}$ are uniformly bounded by some constant $M$. Hence the objective value of this feasible distribution differ from the optimal value by at most $O(\frac{1}{N})$.
      \qed
    \end{proof}

    \begin{theorem}\label{thm:q=infty}
      Suppose $q=\infty$. 
      \begin{enumerate}[label=(\roman*)]
        \item \label{item:dual_infty} Problem \eqref{problem:copula} can be reformulated as
          \[
            %\sup_{F_{k}\in\frakF_{k},\forall k} 
            \min_{\alpha\geq0} \bigg\{ \alpha R_{0}^{p} +  \int_{\Xi} \max_{\xi\in\Xi} \big\{ \Psi(\xi)- \alpha d^{p}(\xi,\zeta):\ \sfd_{F}(\xi,\zeta)\leq r_{0} \big\} \bnu(d\zeta) \bigg\}.
          \]
        \item \label{itm:data-driven} For data-driven problem, i.e., $\bnu=\frac{1}{N}\sum_{i=1}^{N} \bdelta_{\hxi^{i}}$, the following program
          \begin{equation}\label{eqn:data-driven_infty}
            \begin{aligned}
            v_N:=\sup_{\xi^{i}\in\Xi} \bigg\{\frac{1}{N}\sum_{i=1}^{N} \Psi(\xi^{i}):\ &\frac{1}{N}\sum_{i=1}^{N} d^{p}(\xi^{i},\hxi^{i})\leq R_{0}^{p},\\ & \sfd(F(\xi^{i}),F(\hxi^{i})) \leq r_{0},\ \forall 1\leq i\leq N\bigg\}
            \end{aligned}
          \end{equation}
          is a $(1-O(\frac{1}{N}))$-approximation of the inner maximization of \eqref{problem:copula}. In particular when $\sfd(u,v)=||u,v||_{\infty}$, the second constraint in \eqref{eqn:data-driven_infty} can be written as linear constraint
          \begin{equation}\label{eqn:box}
            \begin{aligned}
              \xi^{i}_{k} \ \leq \ \hxi^{j}_{k}, \quad \forall \ i, j \textnormal{ such that } F^{-}_{k}(\hxi^{j}_{k}) - F_{k}(\hxi^{i}_{k}) \geq r_{0}, \\
				\xi^{i}_{k} \ \geq \ \hxi^{j}_{k}, \quad \forall \ i, j \textnormal{ such that } F^{-}_{k}(\hxi^{i}_{k}) - F_{k}(\hxi^{j}_{k}) \geq r_{0},
            \end{aligned}
          \end{equation}
          where $F^{-}_{k}(\hxi_{k}) := \lim_{\xi_{k} \uparrow \hxi_{k}} F_{k}(\xi_{k})$ denotes the left limit of $F_{k}$ at $\hxi_{k}$.
          If, in addition, $\Xi$ is convex and $\Psi$ is a concave function, then the approximation is exact.
      \end{enumerate}
    \end{theorem}

    \begin{proof}
      We first prove \ref{item:dual_infty}.
      For each $\zeta\in\Xi$, set $B_\zeta(r_{0}):=\{\xi\in\Xi:\ \sfd_{F}(\xi,\zeta)\leq r_{0}\}$. It follows that $B_\zeta(r_{0})$ is closed. Using the same reasoning as in the proof of Theorem \ref{thm:copula}, we can obtain the strong duality
      \[
        \begin{aligned}
          &\sup_{\bmu\in\frakM_{2}} \bigg\{\int_{\Xi} \Psi(\xi) \bmu(d\xi)\bigg\}\\
          = & \sup_{\{\bgamma_\zeta\}_\zeta\subset\cP(\Xi)} \inf_{\alpha,\beta\geq0}\bigg\{\int_{\Xi} \Psi(\xi) \bgamma_\zeta(d\xi)\bnu(d\zeta)  + \alpha R_{0}^{p} - \alpha\int_{\Xi} d^{p}(\xi,\zeta)\bgamma_\zeta(d\xi)\bnu(d\zeta):\\
          & \hspace{200pt} \supp\gamma_\zeta \subset B_\zeta(r_{0}),\ \forall \zeta\in\Xi \bigg\}\\
          = & \min_{\alpha\geq0} \bigg\{ \alpha R_{0}^{p} + \max_{\bgamma_\zeta\in\cP(B_\zeta(r_{0}))}\! \int_{\Xi}  \big[\Psi(\xi)  - \alpha d^{p}(u,v) ]\bgamma_\zeta(d\xi)\bnu(d\zeta) \bigg\}.
         \end{aligned}
      \]
      Thus we obtain \ref{item:dual_infty}.

      The first part of statement \ref{itm:data-driven} follows essentially the same as in the proof of Theorem \ref{thm:copula}\ref{itm:approx<infty}. When $\Xi$ is convex and $\sfd$ is induced from $||\cdot||_{\infty}$, the reformulation of the second constraint follows by definition of $\sfd_{F}$, and in this case, $B_\zeta(r_{0})$ is a cube and thus convex. 
      If, in addition, $\Psi$ is concave, then using the same reasoning as in Remark \ref{rmk:concave} we obtain the result.
      \qed
    \end{proof}

  \subsection{Comparison of \eqref{problem:moment} and \eqref{problem:copula}}\label{sec:compare_linear_copula}
    It may appear that formulation \eqref{problem:moment} and \eqref{problem:copula} are in different nature and should have different scopes of application, in the sense that \eqref{problem:moment} considers linear dependence, and \eqref{problem:copula} controls rank dependence as $\{F_{k}(\hxi^{i}_{k})\}$ is an ordinal scaling of the data $\{\hxi^{i}_{k}\}$.
    However, we point out this is not the case, since our new formulation controls not only ordinal association, but to some extent, also the dependence without scaling.
    This is because for relevant distribution $\bmu\in\frakM_{2}$, instead of using its own marginal distributions to make the scaling, we use $F_{k}$, the marginal distribution of the nominal distribution $\bnu$.
    Consequently, the constraint $\sfW_{q}(\bmu,\bnu)\leq r_{0}$ also controls the values, not only the ordinal relationship, of the worst-case distribution. 
    For instance, the box uncertainty set in Example \ref{eg:box} and Theorem \ref{thm:q=infty}\ref{itm:data-driven} confines the region where each data point can be perturbed.
    The following example shows that the worst-case distribution yielding from \eqref{problem:copula} has both similar linear and rank correlation to those of the nominal distribution, whereas \eqref{problem:moment} even does not control the linear correlation, since the one-sided moment constraint may be not tight.

    \begin{example}\label{eg:one-sided}
      Consider a concave objective function
      \[
        \sup_{\bmu\in\frakM_{1}} \E_{\bmu}[-\bxi_{1}^{2}-\bxi_{1}\bxi_{2}-\bxi_{2}^{2}], 
      \]
      where $\Xi=\R^{2}$ equipped with $\ell_{1}$-norm, $\bnu$ is the empirical distribution i.i.d.~drawn from a normal distribution $N\Big([1,1],\Big[\begin{array}{cc}0.3&0\\0&0.3\end{array}\Big]\Big)$, and $R_{0}=0.3$.
      A tractable reformulation based on Theorem \ref{thm:moment} is provided in the appendix.
      The empirical distribution and the worst-case distributions yielding from \eqref{problem:moment} are shown in Fig. \ref{fig:quad_linearCRSO}.
      \begin{figure}
        \centering
        \includegraphics[width=0.5\textwidth]{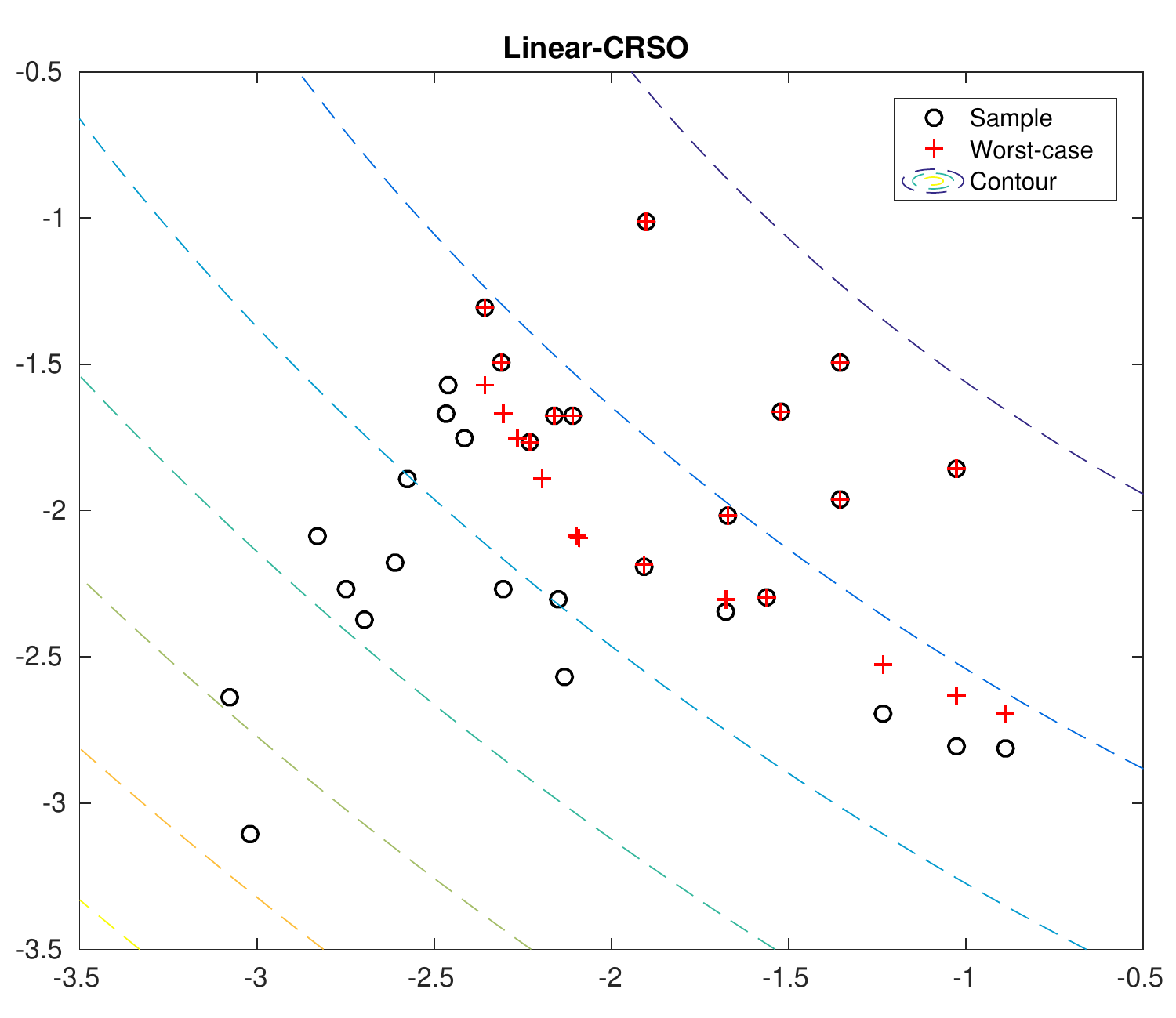}
        \caption{One-sided second-order moment constraint may be not tight}\label{fig:quad_linearCRSO}
      \end{figure}
      
      For the worst-case distribution, the moment constraint in \eqref{problem:moment} is loose, with left-hand side equal to $\Big[\begin{array}{cc}0.216&-0.056\\-0.056&0.150\end{array}\Big]$, less than the sample covariance matrix $\Big[\begin{array}{cc}0.360&-0.008\\-0.008&0.235\end{array}\Big]$.
      The correlation of the worst-case distribution is -0.491, whereas the sample correlation is -0.029. Thereby the correlation structures of the worst-case distribution differs from that of the nominal distribution a lot.
      The intuition is as follows.
      The worst-case distribution of Wasserstein uncertainty set tends to perturb the sample points with large gradient, and since the objective function is concave, the perturbation often leads to a distribution with smaller variance.
      Hence the second-order moment constraint on the main diagonal elements is much smaller than the right-hand side, which gives more flexibility on the off-diagonal elements.
      More specifically, let $\Sigma_{0} = \bigg[\begin{array}{cc}\sigma_{1}^{2} & \rho_{0}\sigma_{1}\sigma_{2} \\ \rho_{0}\sigma_{1}\sigma_{2} & \sigma_{2}^{2} \end{array}\bigg]$, and suppose the variance of the worst-case distribution is $[(1-\epsilon_{1})\sigma_{1}^{2},(1-\epsilon_{2})\sigma_{2}^{2}]$, and let $\rho$ be the correlation of the worst-case distribution. Then the covariance matrix of the worst-case distribution is written as
      $\bigg[\begin{array}{cc} (1-\epsilon)\sigma_{1}^{2} & \rho\sqrt{(1-\epsilon_{1})(1-\epsilon_{2})}\sigma_{1}\sigma_{2} \\ \sqrt{(1-\epsilon_{1})(1-\epsilon_{2})}\rho\sigma_{1}\sigma_{2} & (1-\epsilon)\sigma_{2}^{2} \end{array}\bigg]$. To satisfy the moment constraint, the correlation $\rho$ of the worst-case distribution must satisfy $|\rho-\rho_{0}| \leq \frac{\epsilon}{1-\epsilon}$. When $\epsilon$ is not so small, $\rho$ could have a large discrepancy comparing with the nominal correlation $\rho_{0}$.

      In contrast, if we use reformulation \eqref{eqn:data-driven_infty} of \ref{problem:copula} with $r_{0}=0$, the worst-case distribution is shown in Fig. \ref{fig:quad_rankCRSO}.
      It has the same rank correlation as the nominal distribution, and its linear correlation coefficient equals -0.072.
      \begin{figure}
        \centering
        \includegraphics[width=0.5\textwidth]{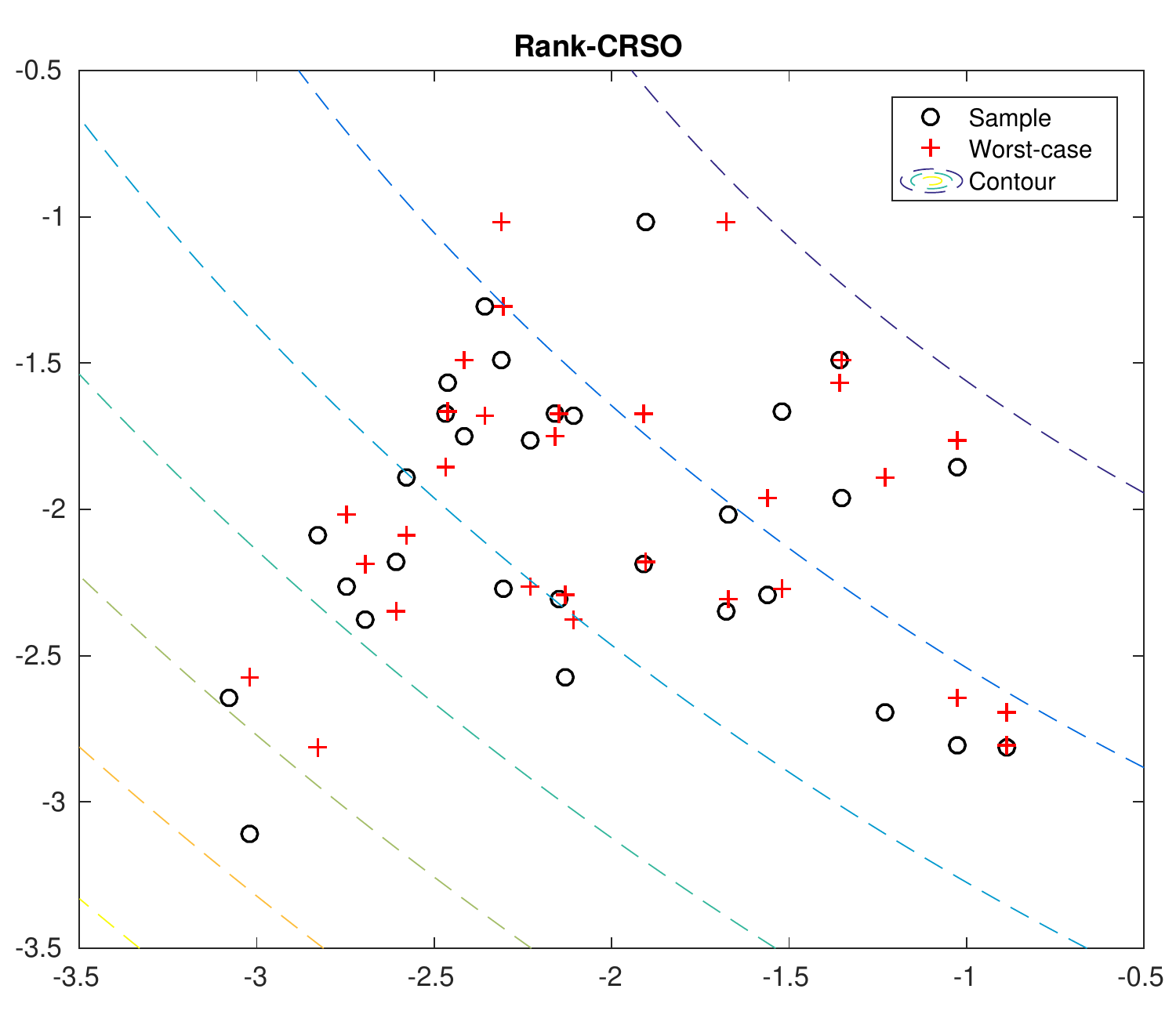}
        \caption{\eqref{problem:copula} controls both rank and linear dependence}\label{fig:quad_rankCRSO}
      \end{figure}
      \qed
    \end{example}

  \subsection{Comparison of Wasserstein distance with other distances}\label{sec:compare_wasserstein}
    In Example \ref{eg:copula}, we show that for continuous distribution, the constraint $\sfW_{q}(\bmu,\bnu)\leq r_{0}$ controls the Wasserstein distance between the copula of $\bmu$ and $\bnu$. 
    Then one may ask consider some other distances and formulate a problem similar to \eqref{problem:copula}.
    We point out that comparing to many other commonly used distances, Wasserstein metrics yields a more intuitive quantitative relationship for copulas of distribution with highly correlated components, as shown by the following example.

    \begin{example}[Distances between Gaussian copulas]
      Let $R\in[-1,1]^{K\times K}$ be a correlation matrix. The Gaussian copula with parameter $R$ is given by
      \[
        \bC_{R}(u) = \phi_{R}(\phi^{-1}(u_{1}),\ldots,\phi^{-1}(u_{d})),
      \]
      where $\phi$ is the cumulative distribution function of a standard Gaussian distribution, and $\phi_{R}$ is the cumulative distribution function of a multivariate Gaussian distribution $\mathcal{N}(0,R)$.
      Now let $\bmu_{i}$, $i=1,2,3$ be three bivariate Gaussian distributions, each of which has standard Gaussian marginal distributions, and the correlation matrices are respectively $R_{1}=\Big[\begin{array}{cc}1&0.5\\0.5&1\end{array}\Big]$, $R_{2}=\Big[\begin{array}{cc}1&0.99\\0.99&1\end{array}\Big]$, and $R_{3}=\Big[\begin{array}{cc}1&0.9999\\0.9999&1\end{array}\Big]$. Their copulas are denoted by $\bC^{\bmu_{i}}$ respectively. 
      Various distances between $\bC^{\bmu_{1}}$ and $\bC^{\bmu_{2}}$ and between $\bC^{\bmu_{2}}$ and $\bC^{\bmu_{3}}$ are shown in Table \ref{tab:copula_distance}, in which total variation distance and $W_{2}$-Wasserstein metric are computed numerically, while the others use closed-form formulas (cf. \cite{marti2016optimal}).

      \begin{table}
      \caption{Distances between bivariate Gaussian copulas}
      \label{tab:copula_distance}
        \begin{tabular}{c|ccccccc}
          \toprule
          Distances & Fisher-Rao & KL & Burg & Hellinger & Bhattacharya & TV & $W_{2}$\\  \midrule
          $\bC^{\bmu_{1}}$, $\bC^{\bmu_{2}}$ & 2.77 & 22.56 & 1.48 & 0.69 & 0.65 & 2.45 & 0.15\\
          $\bC^{\bmu_{2}}$, $\bC^{\bmu_{3}}$ & 3.26 & 47.20 & 1.81 & 0.75 & 0.81 & 4.42 & 0.03\\\bottomrule
        \end{tabular}
      \end{table}
      
      Intuitively, distance between $\bmu_{2}$ and $\bmu_{3}$ should be smaller since both $\bmu_{2}$ and $\bmu_{3}$ are close to a comonotone relationship. Among these distances above, only Wasserstein metric is consistent with our intuition. Therefore, this example renders another justification of the usefulness of our formulation \eqref{problem:copula}.
      \qed
    \end{example}

\section{Concluding remarks}
  In this paper, motivated from the drawback of DRSO with moment uncertainty set, we proposed two new formulations, \ref{problem:moment} and \ref{problem:copula}, that capture different dependence structure for DRSO problems, and derive their dual reformulation. In particular, the dual reformulation is tractable for data-driven \ref{problem:moment} and \ref{problem:copula} with $\infty$-Wasserstein distance. 
  An interesting future direction is to find efficient algorithms for \ref{problem:copula} with $q$-Wasserstein distance ($q\in[1,\infty)$).
  Moreover, various examples and numerical results demonstrate the flexibility and usefulness of our new formulations.

\appendix
\section{Appendix}  
  \subsection{Dual reformulation in Example \ref{eg:one-sided}}
    Let $A$ be a positive semidefinite matrix. Then $\sup_{\bmu\in\frakM_{1}} \E_{\bmu}[-\bxi^{\top} A\bxi]$ admits a strong dual reformulation
    \[
      \begin{aligned}
      \min_{\substack{x\in X,y_{i}\in\R\\\lambda\geq0,\Lambda\succeq 0}}\quad & \lambda R_{0}^{2} + \langle \Lambda,\Sigma_{0} \rangle + \frac{1}{N}\sum_{i=1}^{N} y_{i}\\
      s.t. \quad & \left[\begin{array}{cc} \Lambda+ A & \frac{1}{2}\zeta - \Lambda m_{0} \\ (\frac{1}{2}\zeta - \Lambda m_{0})^{\top} &\  -\zeta^{\top}\hxi^{i} + m_{0}^{\top}\Lambda m_{0}+y_{i} \end{array}\right]\succeq0,\\
      &\hspace{90pt} \forall \ 1\leq i\leq N,1\leq j\leq J.
    \end{aligned}
    \]
    \begin{proof}
      The result is a consequence of Theorem \ref{thm:moment}.
      Observe that $||\xi-\hxi^{i}||=\sup_{||\zeta||_{\ast}\leq1} \zeta^{\top}(\xi-\hxi^{i})$, then by convex programming duality, we have that
      \[\begin{aligned}
        &\max_{\xi\in\Xi} \Big\{-\xi^{\top} A\xi - \lambda||\xi-\hxi^{i}|| - (\xi-m_{0})^{\top}\Lambda(\xi-m_{0})\Big\}\\
        =& \max_{\xi\in\Xi} \inf_{||\zeta||_{\ast}\leq 1} \Big\{-\xi^{\top} A\xi - \lambda\zeta^{\top}(\xi-\hxi^{i}) - (\xi-m_{0})^{\top}\Lambda(\xi-m_{0})\}\\
        =& \inf_{||\zeta||_{\ast}\leq 1}\max_{\xi\in\Xi} \Big\{-\xi^{\top} A\xi - \lambda\zeta^{\top}(\xi-\hxi^{i}) - (\xi-m_{0})^{\top}\Lambda(\xi-m_{0})\}.
        \end{aligned}
      \]
      Hence the constraint
      \[
        y_{i} \geq \max_{\xi\in\Xi} \Big\{-\xi^{\top} A\xi - \lambda||\xi-\hxi^{i}|| - (\xi-m_{0})^{\top}\Lambda(\xi-m_{0})\Big\}
      \]
      can be written as
      \[
        \exists ||\zeta||_{\ast}\leq 1, \ s.t.\ y_{i} \geq \max_{\xi\in\Xi} \Big\{-\xi^{\top} A\xi - \lambda\zeta^{\top}(\xi-\hxi^{i}) - (\xi-m_{0})^{\top}\Lambda(\xi-m_{0})\Big\},
      \]
      which is further equivalent to
      \[
        \exists ||\zeta||_{\ast} \leq 1, \ s.t.
        \left[\begin{array}{cc} \Lambda+A & \frac{1}{2}\lambda\zeta - \Lambda m_{0} \\ (\frac{1}{2}\lambda\zeta - \Lambda m_{0})^{\top} & -\lambda\zeta^{\top}\hxi^{i} + m_{0}^{\top}\Lambda m_{0}+y_{i} \end{array}\right]\!\succeq\!0.
      \]
      Replacing $\lambda\zeta$ by $\zeta$ we obtain the result. \qed
    \end{proof}

%\begin{acknowledgements}
%\end{acknowledgements}

\bibliographystyle{spmpsci}      % mathematics and physical sciences
\bibliography{ref_dependence}   % name your BibTeX data base

\begin{thebibliography}{10}
\providecommand{\url}[1]{{#1}}
\providecommand{\urlprefix}{URL }
\expandafter\ifx\csname urlstyle\endcsname\relax
  \providecommand{\doi}[1]{DOI~\discretionary{}{}{}#1}\else
  \providecommand{\doi}{DOI~\discretionary{}{}{}\begingroup
  \urlstyle{rm}\Url}\fi

\bibitem{agrawal2012price}
Agrawal, S., Ding, Y., Saberi, A., Ye, Y.: Price of correlations in stochastic
  optimization.
\newblock Operations Research \textbf{60}(1), 150--162 (2012)

\bibitem{bayraksan2015data}
Bayraksan, G., Love, D.K.: Data-driven stochastic programming using
  phi-divergences.
\newblock Tutorials in Operations Research pp. 1--19 (2015)

\bibitem{ben2013robust}
Ben-Tal, A., Den~Hertog, D., De~Waegenaere, A., Melenberg, B., Rennen, G.:
  Robust solutions of optimization problems affected by uncertain
  probabilities.
\newblock Management Science \textbf{59}(2), 341--357 (2013)

\bibitem{delage2010distributionally}
Delage, E., Ye, Y.: Distributionally robust optimization under moment
  uncertainty with application to data-driven problems.
\newblock Operations research \textbf{58}(3), 595--612 (2010)

\bibitem{diaconis1977spearman}
Diaconis, P., Graham, R.L.: Spearman's footrule as a measure of disarray.
\newblock Journal of the Royal Statistical Society. Series B (Methodological)
  pp. 262--268 (1977)

\bibitem{doan2015robustness}
Doan, X.V., Li, X., Natarajan, K.: Robustness to dependency in portfolio
  optimization using overlapping marginals.
\newblock Operations Research \textbf{63}(6), 1468--1488 (2015)

\bibitem{doan2012complexity}
Doan, X.V., Natarajan, K.: On the complexity of nonoverlapping multivariate
  marginal bounds for probabilistic combinatorial optimization problems.
\newblock Operations research \textbf{60}(1), 138--149 (2012)

\bibitem{dupavcova1987minimax}
Dupa{\v{c}}ov{\'a}, J.: The minimax approach to stochastic programming and an
  illustrative application.
\newblock Stochastics: An International Journal of Probability and Stochastic
  Processes \textbf{20}(1), 73--88 (1987)

\bibitem{erdougan2006ambiguous}
Erdo{\u{g}}an, E., Iyengar, G.: Ambiguous chance-constrained problems and
  robust optimization.
\newblock Mathematical Programming \textbf{107}(1-2), 37--61 (2006)

\bibitem{esfahani2015data}
Esfahani, P.M., Kuhn, D.: Data-driven distributionally robust optimization
  using the {W}asserstein metric: performance guarantees and tractable
  reformulations.
\newblock arXiv preprint arXiv:1505.05116  (2015)

\bibitem{galton1886regression}
Galton, F.: Regression towards mediocrity in hereditary stature.
\newblock The Journal of the Anthropological Institute of Great Britain and
  Ireland \textbf{15}, 246--263 (1886)

\bibitem{gao2016distributionally}
Gao, R., Kleywegt, A.J.: Distributionally robust stochastic optimization with
  wasserstein distance.
\newblock arXiv preprint arXiv:1604.02199  (2016)

\bibitem{goh2010distributionally}
Goh, J., Sim, M.: Distributionally robust optimization and its tractable
  approximations.
\newblock Operations research \textbf{58}(4-part-1), 902--917 (2010)

\bibitem{Jiang2015Data-driven}
Jiang, R., Guan, Y.: Data-driven chance constrained stochastic program.
\newblock Mathematical Programming pp. 1--37 (2015)

\bibitem{kallenberg2006foundations}
Kallenberg, O.: Foundations of modern probability.
\newblock Springer Science \& Business Media (2006)

\bibitem{kendall1948rank}
Kendall, M.G.: Rank correlation methods.  (1948)

\bibitem{marti2016optimal}
Marti, G., Andler, S., Nielsen, F., Donnat, P.: Optimal transport vs.
  fisher-rao distance between copulas for clustering multivariate time series.
\newblock arXiv preprint arXiv:1604.08634  (2016)

\bibitem{pearson1895note}
Pearson, K.: Note on regression and inheritance in the case of two parents.
\newblock Proceedings of the Royal Society of London \textbf{58}, 240--242
  (1895)

\bibitem{popescu2007robust}
Popescu, I.: Robust mean-covariance solutions for stochastic optimization.
\newblock Operations Research \textbf{55}(1), 98--112 (2007)

\bibitem{rockafellar2009variational}
Rockafellar, R.T., Wets, R.J.B.: Variational analysis, vol. 317.
\newblock Springer Science \& Business Media (2009)

\bibitem{scarf1958min}
Scarf, H., Arrow, K., Karlin, S.: A min-max solution of an inventory problem.
\newblock Studies in the Mathematical Theory of Inventory and Production
  \textbf{10}, 201--209 (1958)

\bibitem{shapiro2001duality}
Shapiro, A.: On duality theory of conic linear problems.
\newblock In: Semi-infinite programming, pp. 135--165. Springer (2001)

\bibitem{shapiro2009lectures}
Shapiro, A., Dentcheva, D., Ruszczynski, A.: Lectures on stochastic
  programming, volume 9 of mps/siam series on optimization.
\newblock Philadelphia, PA: SIAM. Modeling and theory  (2009)

\bibitem{sklar1959fonctions}
Sklar, M.: Fonctions de r{\'e}partition {\`a} n dimensions et leurs marges.
\newblock Universit{\'e} Paris 8 (1959)

\bibitem{spearman1904proof}
Spearman, C.: The proof and measurement of association between two things.
\newblock The American journal of psychology \textbf{15}(1), 72--101 (1904)

\bibitem{sun2015convergence}
Sun, H., Xu, H.: Convergence analysis for distributionally robust optimization
  and equilibrium problems.
\newblock Mathematics of Operations Research  (2015)

\bibitem{wang2016likelihood}
Wang, Z., Glynn, P.W., Ye, Y.: Likelihood robust optimization for data-driven
  problems.
\newblock Computational Management Science \textbf{13}(2), 241--261 (2016)

\bibitem{vzavckova1966minimax}
{\v{Z}}{\'a}{\v{c}}kov{\'a}, J.: On minimax solutions of stochastic linear
  programming problems.
\newblock {\v{C}}asopis pro p{\v{e}}stov{\'a}n{\'\i} matematiky \textbf{91}(4),
  423--430 (1966)

\bibitem{zalinescu2002convex}
Zalinescu, C.: Convex analysis in general vector spaces.
\newblock World Scientific (2002)

\bibitem{zymler2013distributionally}
Zymler, S., Kuhn, D., Rustem, B.: Distributionally robust joint chance
  constraints with second-order moment information.
\newblock Mathematical Programming \textbf{137}(1-2), 167--198 (2013)

\end{thebibliography}

\end{document}